\documentclass{amsart}
\usepackage{amsmath}
\usepackage{paralist}
\usepackage{amsfonts}
\usepackage{amssymb}
\usepackage{amsthm}
\usepackage{amscd}
\usepackage{float}
\usepackage{tikz}
\usepackage{graphicx}
\usepackage[colorlinks=true]{hyperref}
\hypersetup{urlcolor=blue, citecolor=red}
\usepackage{hyperref}

\newtheorem{theorem}{Theorem}[section]
\newtheorem{corollary}[theorem]{Corollary}

\newtheorem{lemma}[theorem]{Lemma}
\newtheorem{proposition}[theorem]{Proposition}

\theoremstyle{definition}
\newtheorem{definition}[theorem]{Definition}
\newtheorem{remark}[theorem]{Remark}

\newcommand{\R}{\mathbb{R}}
\newcommand{\Z}{\mathbb{Z}}
\newcommand{\N}{\mathbb{N}}
\newcommand{\C}{\mathbb{C}}

\makeatletter
\@namedef{subjclassname@2020}{%
  \textup{2020} Mathematics Subject Classification}
\makeatother

\begin{document}

\title[Strichartz estimates for Lipschitz coefficients]{Strichartz estimates for equations with structured Lipschitz coefficients}

\author{Dorothee Frey \and
        Robert Schippa
}

\address{Fakult\"at f\"ur Mathematik, Karlsruher Institut f\"ur Technologie,
Englerstrasse 2, 76131 Karlsruhe, Germany}
\email{dorothee.frey@kit.edu}
\email{robert.schippa@kit.edu}


\maketitle

\begin{abstract}
Sharp Strichartz estimates are proved for Schr\"odinger and wave equations with Lipschitz coefficients satisfying additional structural assumptions. We use Phillips functional calculus as a substitute for Fourier inversion, which shows how dispersive properties are inherited from the constant coefficient case. Global Strichartz estimates follow provided that the derivatives of the coefficients are integrable. The estimates extend to structured coefficients of bounded variations. As applications we derive Strichartz estimates with additional derivative loss for wave equations with H\"older-continuous coefficients and solve nonlinear Schr\"odinger equations. Finally, we record spectral multiplier estimates, which follow from the Strichartz estimates by well-known means.

\end{abstract}

\section{Introduction and Main results}
In the following we show Strichartz estimates for Schr\"odinger and (half-)wave equations with time-independent Lipschitz coefficients under additional structural assumptions. Let $d \geq 1$, and $a_i \in C^{0,1}(\R)$ satisfy an ellipticity condition for $i=1,\ldots,2d$:
\begin{equation}
\label{eq:EllipticityIntroduction}
\exists \lambda, \Lambda > 0: \, \forall x \in \R: \; \lambda \leq a_i(x) \leq \Lambda, \quad i \in \{1,\ldots,2d\}.
\end{equation}
 We consider the Dirac operators
\begin{equation*}
iD_{a_j} = 
\begin{pmatrix}
0 & - i a_{j+d}(x_j) \partial_j ( \cdot) \\
i a_{j}(x_j) \partial_j & 0
\end{pmatrix}
.
\end{equation*}
We further write 
\begin{equation*}
L = D_a^2 = \sum_{j=1}^d D_{a_j}^2 = 
\begin{pmatrix}
- \sum_{j=1}^d a_{j+d}(x_j) \partial_j (a_j(x_j)  \partial_j) & 0 \\
0 & - \sum_{j=1}^d a_{j}(x_j) \partial_j (a_{j+d}(x_j) \partial_j)
\end{pmatrix}
,
\end{equation*}
$|D_L| = L^{\frac{1}{2}}$, and $|D|= |\nabla|$. We consider the homogeneous Schr\"odinger
\begin{equation}
\label{eq:SEQ}
\left\{ \begin{array}{cl}
i \partial_t u + L u &= 0, \quad (t,x) \in \R \times \R^d, \\
u(0) &= u_0 \in H^s(\R^d;\C^2)
\end{array} \right.
\end{equation}
and half-wave equation
\begin{equation}
\label{eq:HalfWaveEq}
\left\{ \begin{array}{cl}
i \partial_t u + L^{\frac{1}{2}} u &= 0, \quad (t,x) \in \R \times \R^d, \\
u(0) &= u_0 \in H^s(\R^d;\C^2).
\end{array} \right.
\end{equation}

The homogeneous Strichartz estimates quantify dispersive effects by estimates
\begin{equation}
\label{eq:HomogeneousStrichartz}
\| u \|_{L^p([0,T],L^q(\R^d))} \lesssim_{T,d,p,q,s} \| u_0 \|_{H^s(\R^d)}
\end{equation}
for solutions to \eqref{eq:SEQ} or \eqref{eq:HalfWaveEq}. In the constant-coefficient case, i.e., $a_i = 1$ for $i=1,\ldots,2d$ global Strichartz estimates
\begin{equation*}
\| u \|_{L^p(\R, L^q(\R^d))} \lesssim_{d,p,q} \| u_0 \|_{\dot{H}^s(\R^d)}
\end{equation*}
hold true with $s$ determined by scaling, see below. 
Our first result are local-in-time Strichartz estimates for $L$ and $L^{\frac{1}{2}}$.

\begin{theorem}
\label{thm:StrichartzEstimate}
Let $d \geq 1$, $2 \leq p \leq \infty$, $2 \leq q < \infty$, $s_\ell = d \big(\frac{1}{2} - \frac{1}{q} \big) - \frac{\ell}{p} $, $\ell \in \{1,2 \}$. Suppose that $a_i \in C^{0,1}(\R)$, $i=1,\ldots,2d$ satisfy \eqref{eq:EllipticityIntroduction}. Then, we find the half-wave Strichartz estimate to hold
\begin{equation}
\label{eq:WaveStrichartz}
\| |D|^{-s_1} e^{it L^{\frac{1}{2}}} u_0 \|_{L^p([0,T],L^q(\R^d))} \lesssim \mu^{\frac{1}{p}} \| u_0 \|_{L^2(\R^d)}
\end{equation}
with $\mu = T \max_{i} \| a_i \|_{\dot{C}^{0,1}} \geq 1$ provided that $\frac{2}{p} + \frac{d-1}{q} = \frac{d-1}{2}$.\\ 
Furthermore, the Schr\"odinger Strichartz estimate holds true
\begin{equation}
\label{eq:SEQStrichartz}
\| |D|^{-s_2} e^{it L} u_0 \|_{L^p([0,T],L^q(\R^d))} \lesssim \mu^{\frac{1}{p}} \| u_0 \|_{H^{\frac{1}{p}}(\R^d)}
\end{equation}
with $\mu = T \max_{i} \| a_i \|^2_{\dot{C}^{0,1}} \geq 1$ provided that $\frac{2}{p} + \frac{d}{q} = \frac{d}{2}$.
\end{theorem}
\begin{remark}
For $q=\infty$, the above estimates remain true after changing to Besov norms $\dot{B}^{-s_\ell,q}_2$, except at the double endpoints $(p,q,d) \neq (2,\infty,2)$ for the Schr\"odinger equation or $(p,q,d) \neq (2,\infty,3)$ for the wave equation.
\end{remark}

If the coefficients have integrable derivatives with small $L^1$-norm, we can show global estimates:
\begin{theorem}
\label{thm:GlobalEstimates}
Let $d \geq 1$, $2 \leq p,q \leq \infty$, $s_{\ell} = d \big( \frac{1}{2} - \frac{1}{q} \big) - \frac{\ell}{p}$, $\ell \in \{1,2\}$, and $a_i \in C^{0,1}(\R)$, $i=1,\ldots,d$ satisfy \eqref{eq:EllipticityIntroduction}. Let $a_{i} = 1$ for $i=d+1,\ldots,2d$. Suppose that $Var(\log( (a_i )) < 2 \pi$. Then, we find the following estimate to hold
\begin{equation}
\label{eq:GlobalWaveStrichartz}
\| |D|^{-s_1} e^{it L^{\frac{1}{2}}} u_0 \|_{L^p(\R;L^q(\R^d))} \lesssim \| u_0 \|_{L^2(\R^d)}
\end{equation}
provided that $\frac{2}{p} + \frac{d-1}{q} = \frac{d-1}{2}$ and $(p,q,d) \neq (2,\infty,3)$.

Furthermore, the Schr\"odinger Strichartz estimate holds true
\begin{equation}
\label{eq:GlobalSEQStrichartz}
\| |D|^{-s_2}  e^{it L} u_0 \|_{L^p(\R;L^q(\R^d))} \lesssim \| u_0 \|_{L^2(\R^d)}
\end{equation}
provided that $\frac{2}{p} + \frac{d}{q} = \frac{d}{2}$ and $(p,q,d) \neq (2,\infty,2)$.
\end{theorem}

Recall that $p \geq 2$ and $q \geq 2$ are necessary due to convolution structure. The Knapp counterexample for constant coefficients gives the necessary conditions for the integrability indices. For the Schr\"odinger equation, this reads
\begin{equation}
\label{eq:SEQStrichartzPairs}
\frac{2}{p} + \frac{d}{q} \leq \frac{d}{2},
\end{equation}
and for the half-wave equation, this is
\begin{equation}
\label{eq:WEStrichartz}
\frac{2}{p} + \frac{d-1}{q} \leq \frac{d-1}{2}.
\end{equation}
Estimates \eqref{eq:WaveStrichartz}-\eqref{eq:GlobalSEQStrichartz} follow for the strict inequalities in \eqref{eq:SEQStrichartzPairs} or \eqref{eq:WEStrichartz}, respectively, by Sobolev embedding for $q \neq \infty$.
The double endpoints $(p,q) = (2,\infty)$ in two dimensions for the Schr\"odinger equation and three dimensions for the wave equation are ruled out by more sophisticated counterexamples due to Montgomery--Smith \cite{MontgomerySmith1998} and E. Stein (cf. \cite[p.~81]{Tao2006}), respectively. Tupels $(s,p,q,d)$, for which the necessary conditions hold, will be referred to as Schr\"odinger or wave Strichartz pairs, respectively. If equality holds in \eqref{eq:SEQStrichartzPairs} or \eqref{eq:WEStrichartz}, the pairs are referred to as sharp.

 Clearly, on a finite time interval we can use H\"older in time and Bernstein's inequality to estimate low frequencies. Hence, on a finite time interval we can as well consider inhomogeneous Sobolev spaces. The estimates are named after Strichartz's pioneering work \cite{Strichartz1977} on constant coefficients, where the relation with $L^2$-Fourier restriction was established (cf. \cite{Tomas1975}). Ginibre--Velo \cite{GinibreVelo1979} covered a wider range of integrability indices, and finally Keel--Tao \cite{KeelTao1998} covered the time-integrability $p=2$ endpoints.

\medskip 

It was clarified in \cite{KeelTao1998} that Strichartz estimates follow from a dispersive estimate and an energy estimate, see Theorem \ref{thm:KeelTao}. In the constant-coefficient case the required dispersive estimate reads as
\begin{equation*}
\| P_1 e^{it (-\Delta)^{k/2}} u_0 \|_{L^\infty(\R^d)} \lesssim (1+|t|)^{-\sigma(k)} \| u_0 \|_{L^1(\R^d)} \quad (t \neq 0),
\end{equation*}
where $P_1$ denotes a smooth frequency projection to unit frequencies, $\sigma(k)$ denotes a decay parameter, and the energy estimate is given by
\begin{equation*}
\| P_1 e^{it (-\Delta)^{k/2}} u_0 \|_{L^2(\R^d)} \lesssim \| u_0 \|_{L^2(\R^d)}.
\end{equation*}
By the interpolation arguments due to Keel--Tao \cite{KeelTao1998} these yield Strichartz estimates for unit frequencies. The claim follows by rescaling for any dyadic frequency range and the dyadic frequency pieces are assembled by Littlewood-Paley theory.

We follow this strategy also in the current setup of Lipschitz coefficients as square function estimates and scaling symmetries are still available. The key step remains the proof of the dispersive estimate at unit frequencies. In the constant-coefficient case the crucial kernel estimate is a consequence of stationary phase estimates. Here we can use Phillips functional calculus as substitute for Fourier inversion.

\medskip

Previously, the first author proved fixed-time $L^p$-estimates for wave equations with structured Lipschitz coefficients in joint work with P. Portal \cite{FreyPortal2020}. In \cite{FreyPortal2020} an adapted scale of Hardy spaces was introduced on which the time-evolution is bounded. Wave packet analysis is not required in the following. We believe that it is worthwhile to track the time-evolution of wave packets also in the current setup. 

Without imposing additional structural assumptions, Tataru proved sharp Stri\-ch\-artz estimates for wave equations with rough coefficients in $C^s$, $0<s\leq 2$ in \cite{Tataru2001} (see also \cite{BahouriChemin1998,Klainerman2001,Tataru2002}). As counterexamples due to Smith--Tataru \cite{SmithTataru2002} show, in the general case there is derivative loss for $C^s$-coefficients, $0<s<2$ compared to the constant-coefficient case; for $C^2$-coefficients the usual Strichartz estimates hold true (see also Smith \cite{Smith1998}). In the present work, the additional structural assumptions rule out the trapping examples by Smith--Tataru, which allows us to recover the Strichartz estimates for constant-coefficients.

General Strichartz estimates for Schr\"odinger equations with variable coefficients were firstly derived by Staffilani--Tataru \cite{StaffilaniTataru2002} for $C^2$-coefficients under non-trapping assumptions, see also Burq--G\'erard--Tzvetkov \cite{BurqGerardTzvetkov2004} for estimates on smooth compact manifolds and Marzuola--Metcalfe--Tataru \cite{MarzuolaMetcalfeTataru2008}. 
Moreover, in one spatial dimension Burq--Planchon \cite{BurqPlanchon2006} showed Strichartz estimates only for coefficients with bounded variation; see also Beli--Ignat--Zuazua \cite{BeliIgnatZuazua2016}. It is conceivable that the proof of Burq--Planchon \cite{BurqPlanchon2006} also applies in higher dimensions. There is a huge body of literature on Strichartz estimates for wave or Schr\"odinger equations, local- and global-in-time, for Laplacians associated with a Riemannian metric $g$, and the above list is by no means of exhaustive; see also \cite{BoucletTzvetkov2007,BoucletTzvetkov2008,MetcalfeTataru2012}. However, we are not aware of sharp Strichartz estimates for metrics having regularity below $C^2$ and satisfying structural assumptions, but see \cite{SmithTataru2005} for the quasilinear case.

We shall also discuss inhomogeneous estimates and Strichartz estimates for H\"older coefficients.
A straight-forward consequence of Theorem \ref{thm:StrichartzEstimate} by Duhamel's formula and Minkowski's inequality, which is already useful to handle nonlinear equations, is the following:
\begin{corollary}
\label{cor:StrichartzEnergyEstimate}
Let $\ell \in \{1,2\}$, and $(\rho,p,q,d)$ be sharp wave $(\ell =1)$ or Schr\"odinger $(\ell = 2)$ admissible Strichartz pairs. Then, we find the following estimates to hold:
\begin{equation}
\label{eq:InhomogeneousEstimatesL1L2}
\| |D|^{-\rho} \langle D \rangle^{-\frac{\ell-1}{p}} u \|_{L_t^p([0,T],L^q(\R^d))} \lesssim_{T, \| a_i \|_{\dot{C}^{0,1}}} \| u(0) \|_{L^2(\R^d)} + \| (i \partial_t + L^{\frac{\ell}{2}}) u \|_{L_t^{1}([0,T],L^{2}(\R^d))}.
\end{equation}
\end{corollary}

A standard argument invoking the Christ--Kiselev lemma \cite{ChristKiselev2001} yields more inhomogeneous estimates with precise dependence on time scale and Lipschitz norm of the coefficients. However, this misses endpoint estimates:
\begin{corollary}
\label{cor:InhomogeneousChristKiselev}
Let $\ell \in \{1,2\}$, and $(\rho,p,q,d)$, $(\tilde{\rho},\tilde{p},\tilde{q},d)$ be sharp wave $(\ell = 1)$ or Schr\"odinger $(\ell = 2)$ admissible Strichartz pairs. Suppose that $p<\tilde{p}'$. Then, we find the following estimates to hold:
\begin{equation}
\label{eq:InhomogeneousEstimates}
\begin{split}
&\quad \| |D|^{-\rho} \langle D \rangle^{-\frac{\ell-1}{p}} \int_0^t e^{i(t-s) L^{\frac{\ell}{2}}} F(s) ds \|_{L_t^p([0,T],L^q(\R^d))} \\
 &\lesssim \mu^{\frac{1}{p}+\frac{1}{\tilde{p}}} \| |D|^{\tilde{\rho}} \langle D \rangle^{\frac{\ell-1}{\tilde{p}}} F \|_{L_t^{\tilde{p}'}([0,T],L^{\tilde{q}'}(\R^d))}
 \end{split}
\end{equation}
with $\mu = \max_{i} T \| a_i \|^{\ell}_{\dot{C}^{0,1}}$.
\end{corollary}
More inhomogeneous estimates are available by the bilinear interpolation due to Keel--Tao \cite{KeelTao1998}. The estimates were further refined by Foschi \cite{Foschi2005} (see also \cite{Vilela2007,Taggart2010,Schippa2017,LeeSeo2014}). For the sake of brevity, we only record the estimates due to Keel--Tao, but remark that the estimates from \cite{Foschi2005} hold in a wider range.

\begin{theorem}
\label{thm:GlobalInhomogeneousStrichartzEstimates}
Let $\ell \in \{ 1, 2 \}$, and let $(\rho,p,q,d)$, $(\tilde{\rho},\tilde{p},\tilde{q},d)$ be sharp wave $(\ell = 1)$ or Schr\"odinger $(\ell = 2)$ admissible Strichartz pairs. Let $a_i = 1$ for $i=d+1,\ldots,2d$. Suppose that $Var(\log(a_i)) < 2 \pi$. Then, we find the following inhomogeneous Strichartz estimates to hold:
\begin{equation}
\label{eq:InhomogStrichartzEstimate}
\| |D|^{-\rho} \int_0^t e^{i(t-s) L^{\frac{\ell}{2}}} F(s) ds \|_{L_t^p(\R, L^q_x(\R^d))} \lesssim \| |D|^{\tilde{\rho}} F \|_{L^{\tilde{p}'}_t(\R, L^{\tilde{q}'}_x(\R^d))}.
\end{equation}
\end{theorem}
We remark that in case of the half-wave equation, we can also apply the Keel-Tao argument for finite times. In Section \ref{section:Applications} we give further applications of the analysis. Since the dispersive estimate hinges on integrals of the derivatives, we can extend the results to coefficients with (locally) bounded variation. We refer to Subsection \ref{subsection:BVCoefficients} for details. Next, we use a refined version of Corollary \ref{cor:StrichartzEnergyEstimate} to derive wave Strichartz estimates for H\"older coefficients. We truncate the coefficients in Fourier space, which allows to use Strichartz estimates for Lipschitz coefficients for frequency localized functions and use commutator estimates. Similar arguments were previously used by Tataru \cite{Tataru2001}; see also \cite{BahouriChemin1998}. Below $(S_N)_{N \in 2^{\mathbb{N}_0}}$ denotes an inhomogeneous Littlewood--Paley decomposition in spatial frequencies.
\begin{theorem}
\label{thm:HolderCoefficients}
Let $d \geq 2$, $s \in (0,1)$, $T>0$, $P = \partial_t^2 - \sum_{i=1}^d \partial_{x_i} a_i(x_i) \partial_{x_i}$, and $a_i \in \dot{B}^s_{\infty,1}(\R)$ satisfy \eqref{eq:EllipticityIntroduction} for $i=1,\ldots,d$. Suppose that $T^s \| a_i \|_{\dot{B}^s_{\infty,1}(\R)} \leq \mu$ for some $\mu \geq 1$. Then, we find the following estimate to hold:
\begin{equation}
\label{eq:DyadicEstimateHolderCoefficients}
\sup_{N \in 2^{\mathbb{N}_0}} N^{1-\rho-\frac{\sigma}{p}} \| S_N u \|_{L^p([0,T],L^q(\R^d))} \lesssim \mu^{\frac{1}{p}} \| \nabla u \|_{L^\infty([0,T], L^2(\R^d))} + \mu^{-\frac{1}{p'}} \| |D|^{-\sigma} P u \|_{L^1([0,T],L^2(\R^d))}
\end{equation}
provided that $(\rho,p,q,d)$ is a wave Strichartz pair and $\sigma = 1-s$.
\end{theorem}
We also apply Strichartz estimates to nonlinear Schr\"odinger equations. With the same Strichartz estimates as in case of constant coefficients at hand, the arguments to show analytic well-posedness for (sub-)critical nonlinear Schr\"odinger equations are standard by now (cf. \cite{Tsutsumi1987,Tao2006}). We refer to Subsection \ref{subsection:NonlinearEquations} for details. In Section \ref{section:BochnerRiesz} we point out how the global Schr\"odinger Strichartz estimates of Theorem \ref{thm:GlobalEstimates} yield Bochner--Riesz and spectral multiplier estimates for the scalar operator $L= - \sum_{i=1}^d \partial_{x_i} (a_i(x_i) \partial_{x_i})$ with $a_i \in BV(\R)$ satisfying ellipticity conditions. For this we apply the abstract results due to Chen \emph{et al.} \cite{ChenOuhabazSikoraYan2016,ChenLeeSikoraYan2020} and Sikora--Yan--Yao \cite{SikoraYanYao2018}. 

\medskip

\emph{Outline of the paper.} In Section \ref{section:Preliminaries} we recall basic facts about elliptic operators with Lipschitz coefficients as resolvent estimates and the Phillips functional calculus. In Section \ref{section:StrichartzEstimates} we show Strichartz estimates for Lipschitz coefficients, and in Section \ref{section:Applications} we record applications. We extend the Strichartz estimates to $BV$-coefficients, derive wave Strichartz estimates for H\"older coefficients, and also record well-posedness results for nonlinear Schr\"odinger equations with $BV$-coefficients. In Section \ref{section:BochnerRiesz} we show Bochner--Riesz means and spectral multiplier estimates.

\section{Preliminaries}
\label{section:Preliminaries}

We recall the basic setup of \cite{FreyPortal2020} with an emphasis on the Phillips functional calculus. We also collect basic resolvent estimates and the square function estimate.
Recall that for $i \in \{1,\ldots, 2d\}$, $a_i \in C^{0,1}(\R)$, which satisfy the ellipticity condition
\begin{equation}
\label{eq:Ellipticity}
\exists \lambda, \Lambda > 0: \, \forall x \in \R: \; \lambda \leq a_i(x) \leq \Lambda.
\end{equation}
The ellipticity constants will be fixed for the rest of the paper. 
\begin{definition}
\label{def:DiracOperators}
For $j \in \{1,\ldots,2d\}$, let $a_j \in C^{0,1}(\R)$ satisfy \eqref{eq:Ellipticity}. For $\xi = (\xi_1,\ldots,\xi_d) \in \R^d$, we define
\begin{equation*}
\begin{split}
\xi. D_a &= \sum_{j=1}^d \xi_j \begin{pmatrix}
0 & - a_{j+d}(x_j) \partial_j (  \cdot) \\
a_j(x_j) \partial_j & 0
\end{pmatrix}, \\
\xi. \sqrt{D_a^2} &= \sum_{j=1}^d \xi_j 
\begin{pmatrix}
\sqrt{- a_{j+d}(x_j) \partial_j ( a_j(x_j) \partial_j)} & 0 \\
0 & \sqrt{- a_j(x_j) \partial_j (a_{j+d}(x_j) \partial_j)}
\end{pmatrix}
,
\end{split}
\end{equation*}
which we view as unbounded operator in $L^2(\R^d;\C^2)$ with domain $W^{1,2}(\R^d;\C^2)$.
\end{definition}
The Dirac operators $D_{a_i}$ are not commuting. The remedy is to use the cosine group, which is as well generated by $\sqrt{D_{a_i}^2}$, and the operators $\sqrt{D_{a_i}^2}$ and $\sqrt{D_{a_j}^2}$ are commuting. Recall the following about the generator of the transport group:

\begin{proposition}[{\cite[Proposition~3.1]{FreyPortal2020}}]
\label{prop:PropertiesTransport}
The operator $a \frac{d}{dx}$ generates a bounded $C_0$-group in $L^p(\R)$ for all $p \in [1,\infty)$. Moreover, $\phi(x) = \int_0^x \frac{1}{a(y)} dy$ is a global $C^1$-diffeomorphism, and we have
\begin{equation*}
(e^{ta \partial_x} f)(x) = f(\chi(t,x)) \text{ for } f \in C^\infty_c(\R)
\end{equation*}
with $\chi(t,x) = \phi^{-1}(t+\phi(x))$, $\chi \in C^1(\R^2;\R)$.
\end{proposition}
We can now prove the following basic property of the Dirac operator:

\begin{proposition}
\label{prop:TransportGroupPerturbation}
Let $a$, $b \in C^{0,1}(\R)$ satisfy \eqref{eq:Ellipticity}. Then, the operators
\begin{equation*}
A: W^{1,p}(\R;\C^2) \to L^p(\R;\C^2), \quad A =
\begin{pmatrix}
0 & b \partial_x ( \cdot) \\
a \partial_x & 0
\end{pmatrix}, \; i \tilde{D}_a = i \begin{pmatrix}
0 & - b \partial_x(\cdot) \\
a \partial_x & 0
\end{pmatrix}
\end{equation*}
with the same domain generate $C_0$-groups in $L^p(\R;\C^2)$ for $1 \leq p < \infty$.
\end{proposition}
\begin{proof}
It will be enough to show the claim for $A$ as we shall see that the proof extends to $i \tilde{D}_a$. For the proof we diagonalize $A$ up to an $L^p$ bounded operator. The principal symbol is given by
\begin{equation*}
i \xi
\begin{pmatrix}
0 & b(x) \\
a(x) & 0
\end{pmatrix}
.
\end{equation*}
The eigenvalues are given by $\pm i\xi (ab)^{1/2}(x)$, and it suffices to note a matrix $M$ of eigenvectors is given by
\begin{equation*}
M = \begin{pmatrix}
b^{1/2} & b^{1/2} \\
a^{1/2} & -a^{1/2}
\end{pmatrix},
\quad M^{-1} = \frac{1}{2} \begin{pmatrix}
b^{-1/2} & a^{-1/2} \\
b^{-1/2} & - a^{-1/2}
\end{pmatrix}
.
\end{equation*}
With this it is straightforward that
\begin{equation*}
\begin{split}
&
\frac{1}{2}
\begin{pmatrix}
b^{-1/2}(x) & a^{-1/2}(x) \\
b^{-1/2}(x) & - a^{-1/2}(x)
\end{pmatrix}
\begin{pmatrix}
0 & \partial_x (b(x) \cdot) \\
a(x) \partial_x & 0
\end{pmatrix}
\begin{pmatrix}
b^{1/2}(x) & b^{1/2}(x) \\
a^{1/2}(x) & -a^{1/2}(x)
\end{pmatrix}
\\
&=
\begin{pmatrix}
(ab)^{1/2}(x) \partial_x & 0 \\
0 & -(ab)^{1/2}(x) \partial_x
\end{pmatrix}
+ E(x),
\end{split}
\end{equation*}
where $E$ is a $2 \times 2$ matrix consisting of linear combinations of derivatives of the coefficients, which are bounded. Hence, we can transform the equation
\begin{equation*}
\partial_t
\begin{pmatrix}
u_1 \\ u_2
\end{pmatrix}
= 
\begin{pmatrix}
0 & b \partial_x (\cdot) \\
a \partial_x & 0
\end{pmatrix}
\begin{pmatrix}
u_1 \\ u_2
\end{pmatrix}
\end{equation*}
to
\begin{equation*}
\partial_t
\begin{pmatrix}
v_1 \\ v_2
\end{pmatrix}
=
\left(
\begin{pmatrix}
(ab)^{1/2} \partial_x & 0 \\
0 & - (ab)^{1/2} \partial_x
\end{pmatrix}
+E \right)
\begin{pmatrix}
v_1 \\ v_2
\end{pmatrix}
\end{equation*}
with
\begin{equation}
\label{eq:LpBoundE}
\begin{pmatrix}
v_1 \\ v_2
\end{pmatrix}
= \frac{1}{2}
\begin{pmatrix}
b^{-1/2} & a^{-1/2} \\
b^{-1/2} & -a^{-1/2}
\end{pmatrix}
\begin{pmatrix}
u_1 \\ u_2
\end{pmatrix},
\quad \| E \|_{L^p \to L^p} \leq C(\lambda,\Lambda) \max_{c=a,b} (\| c \|_{\dot{C}^{0,1}}).
\end{equation}
For $v$ we have the following representation by Duhamel's formula:
\begin{equation}
\label{eq:SolutionV}
v = 
\begin{pmatrix}
e^{t(ab)^{1/2}(x) \partial_x} v_1 \\
e^{-t(ab)^{1/2}(x) \partial_x} v_2
\end{pmatrix}
+ \int_0^t T_{t-s} (Ev)(s) ds,
\end{equation}
with $(T_t)$ denoting the $C_0$-transport group generated by 
\begin{equation}
B = \text{diag}((ab)^{1/2}(x) \partial_x, -(ab)^{1/2}(x) \partial_x).
\end{equation}
From \eqref{eq:SolutionV} is immediate that $v_0 \mapsto v(t)$ is a $C_0$-group in $L^p(\R;\C^2)$ for $1 \leq p < \infty$, and so is $u_0 \mapsto u(t)$.
In a similar spirit, we can show that $i \tilde{D}_a$ generates a $C_0$-group in $L^p(\R^2;\C)$ for $1\leq p < \infty$ because the diagonalization still applies, and after change of basis, $i \tilde{D}_a$ generates the transport group up to an $L^p$-bounded error.
\end{proof}
\noindent We find a representation in terms of solutions to wave equations by squaring the transport group. Let
\begin{equation*}
\begin{pmatrix}
u \\
v
\end{pmatrix}
(t) = e^{it D_a} 
\begin{pmatrix}
u_0 \\ v_0
\end{pmatrix}.
\end{equation*}
By taking two time derivatives, we find
\begin{equation*}
\left\{
\begin{array}{cl}
\partial_t^2 u &= b \partial_x ( a \partial_x) u, \\
\partial_t^2 v &= a \partial_x (b \partial_x v).
\end{array} \right.
\end{equation*}
Let $L_1 = - b \partial_x( a \partial_x)$, $L_2 = -a \partial_x (b \partial_x)$. Then we can write the solution as
\begin{equation*}
\left\{ \begin{array}{cl}
u(t) &= \cos(t L_1) u_0 - i \frac{\sin(t L_1)}{L_1} b \partial_x  v_0, \\
v(t) &= \cos(t L_2) v_0 + i \frac{\sin(t L_2)}{L_2} a \partial_x u_0,
\end{array} \right.
\end{equation*}
or, concisely,
\begin{equation*}
e^{it D_a} = 
\begin{pmatrix}
\cos(t L_1) & -i \frac{\sin(t L_1)}{L_1} (b \partial_x) \\
i \frac{\sin(t L_2)}{L_2} a \partial_x & \cos(t L_2)
\end{pmatrix}
.
\end{equation*}
A similar computation yields for
\begin{equation*}
\begin{pmatrix}
u \\ v
\end{pmatrix}
(t) = e^{ it \sqrt{D_a^2}}
\begin{pmatrix}
u_0 \\ v_0
\end{pmatrix}
,
\end{equation*}
the representation
\begin{equation*}
\left\{ \begin{array}{cl}
u(t) &= \cos(t L_1) u_0 + i \sin(t L_1) u_0, \\
v(t) &= \cos(t L_2) v_0 + i \sin(t L_2) v_0,
\end{array} \right.
\end{equation*}
or, concisely,
\begin{equation*}
\begin{pmatrix}
u(t) \\ v(t)
\end{pmatrix}
=
\begin{pmatrix}
\cos(t L_1) + i \sin(t L_1) & 0 \\
0 & \cos(tL_2) + i \sin(tL_2)
\end{pmatrix}
\begin{pmatrix}
u_0 \\ v_0
\end{pmatrix}
= e^{it \sqrt{D_a^2}} 
\begin{pmatrix}
u_0 \\ v_0
\end{pmatrix}
.
\end{equation*}
The Dirac operators $(D_a)$ are bisectorial operators (cf. \cite[Definition~2.4]{FreyPortal2020}, \cite[Chapter~10]{HytoenenNeervenVeraarWeis2017}), but not commuting. However, the cosine groups of the operators coincide:
\begin{equation*}
e^{it \sqrt{D_a^2}} + e^{-it \sqrt{D_a^2}} = e^{it D_a} + e^{-it D_a}.
\end{equation*}
This allows us to recover a Phillips functional calculus, for which we follow the argument in \cite{FreyPortal2020}:
\begin{definition}
\label{def:PhillipsCalculus}
For $\Psi \in \mathcal{S}(\R^d)$, we define $\Psi(\sqrt{D_a^2})$ using Phillips functional calculus associated with the commutative group $\big( \exp (i \xi. \sqrt{D_a^2}) \big)_{\xi \in \R^d}$:
\begin{equation*}
\Psi(\sqrt{D_a^2}) = \frac{1}{(2 \pi)^d} \int_{\R^d} \hat{\Psi}(\xi) \exp(i \xi. \sqrt{D_a^2}) d \xi.
\end{equation*}
We consider functions $\Psi$ which are even in every coordinate, i.e., $\Psi = \Psi^s$ with
\begin{equation*}
\Psi^s(x) = 2^{-d} \sum_{(\delta_j)^d_{j=1} \in \{ -1 , 1 \}^d} \Psi(\delta_1 x_1, \ldots , \delta_d x_d).
\end{equation*}
For such functions, we have that
\begin{equation*}
\begin{split}
&\quad \Psi(\sqrt{D_a^2}) \\
&= \frac{1}{(2 \pi)^d} \int_{\R^{d-1}} \int_{\R} \hat{\Psi}^s(\xi) \frac{1}{2} \big( \exp( i \xi_1 e_1 \sqrt{D_a^2}) + \exp(-i \xi_1 e_1 \sqrt{D_a^2}) \big) \exp(i(\xi - \xi_1 e_1) \sqrt{D_a^2}) \\
&= \frac{1}{(2 \pi)^d} \int_{\R^d} \hat{\Psi}^s(\xi) \prod_{j=1}^d \exp( i \xi_j D_{a_j}) d\xi,
\end{split}
\end{equation*}
since $D_a$ and $\sqrt{D_a^2}$ generate the same cosine family. 
For the sake of brevity, we write
\begin{equation*}
\Psi^s(\sqrt{D_a^2}) = \Psi^s(D_a) = \frac{1}{(2 \pi)^d} \int_{\R^d} \hat{\Psi}^s(\xi) \exp(i \xi. D_a) d\xi.
\end{equation*}
Moreover, we have
\begin{equation*}
L = D_a^2 = \sqrt{D_a^2} \cdot \sqrt{D_a^2} = \begin{pmatrix}
- \sum_{j=1}^d a_{j+d}(x_j)  \partial_{x_j} (a_j(x_j) \partial_{x_j}) & 0 \\
0 & - \sum_{j=1}^d a_j(x_j) \partial_{x_j}  (a_{j+d}(x_j) \partial_{x_j})
\end{pmatrix}.
\end{equation*}
\end{definition}

As another consequence of the close relation with the transport group, the group generated by $i \xi.D_a$ satisfies a strong form of finite speed of propagation (cf. \cite[Remark~4.2]{FreyPortal2020}).
We summarize the finite speed of propagation property in the following lemma:
\begin{lemma}
\label{lem:FiniteSpeedPropagation}
Let $u_0 \in C^\infty_c(\R;\C^2)$, and $\partial_t u  = i \tilde{D}_a u$, $u(0) = u_0$. Then, we find 
\begin{equation*}
u(t,x) = 0 \text{ if dist}(x,\text{supp}(u_0)) \geq C |t|.
\end{equation*}
\end{lemma}

Moreover, after introducing the scalar product in $L^2(\R;\C^2)$,
\begin{equation*}
\langle \begin{pmatrix}
u_1 \\ v_1
\end{pmatrix}, 
\begin{pmatrix}
u_2 \\ v_2
\end{pmatrix}
\rangle_{(a,b)} :=
\langle \begin{pmatrix}
u_1 \\ v_1
\end{pmatrix}
, \begin{pmatrix}
b^{-1} & 0 \\
0 & a^{-1}
\end{pmatrix}
\begin{pmatrix}
u_2 \\ v_2
\end{pmatrix}
\rangle,
\end{equation*}
we find $iD_a$ to be a self-adjoint operator, which implies the following global $L^2$-estimates.
\begin{proposition}
\label{prop:L2Estimate}
Let $u_0 \in L^2(\R;\C^2)$ and $\partial_t u = i \tilde{D}_a u$, $u(0) = u_0$. Then, we find the following estimate to hold:
\begin{equation}
\label{eq:L2Estimate}
\| u(t) \|_{L^2(\R;\C^2)} \lesssim \| u_0 \|_{L^2(\R;\C^2)}.
\end{equation}
\end{proposition}

In subsequent sections we need the following properties of the operators $D_a$ and $L$:
\begin{proposition}
\label{prop:OperatorNorms}
There is $N=N(d)$ such that we find the following estimate to hold:
\begin{equation}
\label{eq:SobolevEmbedding}
\| (1+L)^{-N} \|_{L^1(\R^d) \to L^\infty(\R^d)} \leq C(\lambda,\Lambda).
\end{equation}
\end{proposition}

We shall also make use of Littlewood-Paley theory associated with $L^{\frac{1}{2}} = |D_L|$. For this purpose, we consider a radially decreasing function 
\begin{equation*}
\phi(\xi) \equiv 1 \text{ for } |\xi| \leq 1, \quad \text{supp} (\phi) \subseteq B(0,2).
\end{equation*}
Let $\psi(\xi) = \phi(\xi/2) - \phi(\xi)$ and $\psi_k(\xi) = \psi(2^{-k} \xi)$ such that
\begin{align*}
\phi(\xi) + \sum_{k \in \N} \psi_k(\xi) = 1, \quad \sum_{k \in \Z} \psi_k(\xi) \equiv 1 \quad (\xi \neq 0).
\end{align*}
In view of the previous paragraph, we note that $\phi = \phi^s$ and $\Psi = \Psi^s$. We have the following square function estimate:
\begin{proposition}
\label{prop:SquareFunctionEstimate}
Let $1 < p < \infty$. Then, we find the following estimate to hold
\begin{equation}
\label{eq:SquareFunctionEstimate}
\| f \|_{L^p(\R^d)} \sim \| \big( \sum_{k \in \Z} | \psi_k(\sqrt{D_a^2}) f |^2 \big)^{\frac{1}{2}} \|_{L^p(\R^d)}
\end{equation}
with implicit constant only depending on $d$, $p$, and $ \| a \|_{C^{0,1}}$.
\end{proposition}
This follows from $L^q$-boundedness of $F(L^{\frac{1}{2}})$ provided that $F$ satisfies the usual Mikhlin condition (cf. \cite{StrkaljWeis2007}). The square function estimate is concluded by a Rademacher function argument. 

\medskip

It turns out that the Besov norms for structured $L^\infty$-coefficients remain equivalent. Let $(a_i)_{i=1,\ldots,d} \subseteq L^\infty(\R)$ satisfy \eqref{eq:Ellipticity}. Let $L=- \sum_{i=1}^d \partial_{x_i} (a_i(x_i) \partial_{x_i})$. Let $S_L(t) = e^{t L}$ denote the associated heat kernel and $\Delta_j^L = 4^{-j} L S_L(4^{-j})$ the dyadic projection. We define Besov spaces associated with $L$ as 
\begin{equation*}
\| f \|_{\dot{B}^{s,p}_{q,L}} = \big( \sum_{j \in \Z} 2^{sqj} \| \Delta_j^L f \|^q_{L^p} \big)^{\frac{1}{q}}
\end{equation*}

\begin{proposition}
\label{prop:EquivalenceBesovNorms}
Under the above assumptions, for $1<p<\infty$ and $-1<s<1$, we have $\dot{B}^{s,p}_q = \dot{B}^{s,p}_{q,L}$ with equivalence of the norms only depending on the ellipticity constants.
\end{proposition}
\begin{proof}
In the present case, the heat kernel factors as a consequence of
\begin{equation*}
[\partial_{x_i} (a_i(x_i) \partial_{x_i}), \partial_{x_j} (a_j(x_j) \partial_{x_j})] = 0.
\end{equation*}
Since the operators $(\partial_{x_i}(a_i(x_i) \partial_{x_i})_{i}$ are commuting, the properties of the one-dimensional heat kernel of $L_i = - \partial_{x_i} (a_i(x_i) \partial_{x_i})$ are inherited for $L$:
\begin{equation*}
K_L(x,y,t) = S_L(t)(x,y) = e^{t L}(x,y) = \prod_{i=1}^d e^{t L_i}(x_i,y_i).
\end{equation*}
The one-dimensional case was discussed in detail in \cite[Appendix~A]{BurqPlanchon2006}.
By the properties of the one-dimensional kernel, there exists $c$ depending only on the ellipticity constants such that
\begin{equation}
\label{eq:HeatKernelEstimate}
|K_L(x,y,t)| \lesssim t^{-\frac{d}{2}} e^{- \frac{|x-y|^2}{t}}.
\end{equation}
Moreover,
\begin{equation*}
|\partial_{y_i} K_L(x,y,t)| + |\partial_{x_i} K_L(x,y,t)| \lesssim t^{-\frac{d}{2}-\frac{1}{2}} e^{-\frac{|x-y|^2}{t}}
\end{equation*}
and
\begin{equation*}
|L K_L(x,y,t)| \lesssim t^{-\frac{d}{2}-1} e^{-\frac{|x-y|^2}{t}}.
\end{equation*}
By the Gaussian bounds, it follows that $S_L(t)$ is continuous on $L^p$, as well as $\Delta_j^L = 4^{-j} L S_L(4^{-j})$. By the kernel estimate, we can derive equivalence of Besov norms as in \cite[Appendix~A.2]{BurqPlanchon2006}.
\end{proof}
\begin{remark}
We remark that the argument strongly hinges upon the structure of the operator. For general elliptic operators with H\"older coefficients the equivalence of Besov norms fails.
\end{remark}
We shall also need the following square function estimate, which even holds for general elliptic $L^\infty$-coefficients (cf. \cite{AuscherHofmannLaceyMcIntoshTchamitchian2002}):
\begin{proposition}[One-dimensional Kato square-root estimate]
\label{prop:SquareRootEstimate}
Under the above assumptions, we have
\begin{equation}
\label{eq:KatoSquareFunctionEstimate}
\| |D_L| f \|_{L^2} \sim_{\lambda,\Lambda} \| f \|_{\dot{H}^1}.
\end{equation}
\end{proposition}

\section{Strichartz estimates for Lipschitz coefficients}
\label{section:StrichartzEstimates}
This section is devoted to the proof of Strichartz estimates for structured Lipschitz coefficients. The key ingredient in our proof are as in the constant-coefficient case dispersive estimates. Recall how dispersive estimates imply Strichartz estimates by the following abstract result due to Keel--Tao \cite{KeelTao1998}:
\begin{theorem}[Keel--Tao] \label{thm:KeelTao}
Let $(X,dx)$ be a measure space and $H$ a Hilbert space. Suppose that for each $t \in \R$ we have an
operator $U(t):H \to L^2(X)$ which satisfies the following assumptions for $\sigma>0$:
\begin{itemize}
\item[(i)] For all $t$ and $f \in H$ we have the energy estimate:
\begin{equation*}
\| U(t) f \|_{L^2(X)} \lesssim \|f \|_H.
\end{equation*}
\item[(ii)] For all $t \neq s$ and $g \in L^1(X)$ we have the decay estimate
\begin{equation*}
\| U(s) (U(t))^* g \|_{L^\infty(X)} \lesssim (1+|t-s|)^{-\sigma} \| g \|_{L^1(X)}.
\end{equation*}
\end{itemize}
Then, for $\sigma$\emph{-admissible exponents} $(p,q)$, which satisfy $\frac{1}{p} + \frac{\sigma}{q} \leq \frac{\sigma}{2}$ and $(p,q,\sigma) \neq (2,\infty,1)$, the estimate
\begin{equation*}
\| U(t) f \|_{L^p_{t} L^q_x(\R \times X)} \lesssim \| f \|_H
\end{equation*}
holds. Furthermore, for two $\sigma$-admissible pairs $(p,q)$ and $(\tilde{p},\tilde{q})$, the estimate
\begin{equation}
\label{eq:InhomogeneousEstimateKT}
\| \int_{s<t} U(t) (U(s))^* F(s) ds \|_{L^p_t L_x^q(\R \times X)} \lesssim \| F \|_{L_t^{\tilde{p}'} L_x^{\tilde{q}'}}
\end{equation}
holds true.
\end{theorem}
We shall see by Littlewood-Paley decomposition and rescaling that it will be enough to show estimates for unit frequencies. The energy estimate required by Theorem \ref{thm:KeelTao} (i)  is Proposition \ref{prop:L2Estimate}. We have to show the dispersive estimate:
\begin{proposition}
\label{prop:DispersiveEstimateUnitFrequencies}
Let $a_i \in C^{0,1}(\R)$ satisfy \eqref{eq:Ellipticity}. Then, we find the following estimate to hold:
\begin{equation}
\label{eq:DispersiveEstimateUnitFrequencies}
\| e^{it L^{\frac{\ell}{2}}} \psi(\sqrt{D_a^2}) \|_{L^1(\R^d) \to L^\infty(\R^d)} \lesssim (1+|t|)^{-\sigma(\ell)}
\end{equation}
for $\ell \in \{1,2\}$, $0 < t \leq C(\lambda,\Lambda) \max (\| a_i \|_{\dot{C}^{0,1}})$, and
\begin{equation*}
\sigma(\ell) = \begin{cases}
\begin{aligned}
&\frac{d-1}{2}, \quad &&\ell = 1, \\
&\frac{d}{2}, \quad &&\ell = 2.
\end{aligned}
\end{cases}
\end{equation*}
\end{proposition}
\begin{proof}
By Phillips functional calculus, we have to prove
\begin{equation}
\label{eq:PhillipsReduction}
\sup_{x \in \R^d} \big| \int K_t(\xi) e^{i \xi.D_a} u_0(x) d\xi \big| \lesssim (1+|t|)^{-\sigma(\ell)} \| u_0 \|_{L^1(\R^d)},
\end{equation}
where
\begin{equation*}
K_t(\xi) = \frac{1}{(2\pi)^d} \int_{\R^d} e^{it|y|^\ell} \psi(y) e^{i y. \xi} dy
\end{equation*}
denotes the kernel of the half-wave equation for $\ell =1 $ or the Schr\"odinger equation for $\ell = 2$ at unit frequencies. Note that $K_t(\xi) = K_t^s(\xi)$ and Definition \ref{def:PhillipsCalculus} applies. Recall that
\begin{equation}
\label{eq:UnitFrequencyKernel}
|K_t(\xi)| \lesssim
\begin{cases}
\begin{aligned}
&(1+|t|+|\xi|)^{-N}, \quad &&|t| \not \sim |\xi|, \\
&(1+|t|)^{-\sigma(\ell)}, \quad &&|t| \sim |\xi|
\end{aligned}
\end{cases}
\end{equation}
by (non-)stationary phase (cf. \cite[Chapter~1]{Sogge2017}). Hence, to show the dispersive estimate \eqref{eq:DispersiveEstimateUnitFrequencies}, it suffices to prove
\begin{equation*}
\int_{\R^d} \big| e^{i \xi.D_a} f(x) \big| d\xi \lesssim \int_{\R^d} |f(x)| dx.
\end{equation*}
We show this after suitable localization.

By the $L^2$-estimate, Phillips functional calculus, and the rapid decay of the kernel \eqref{eq:UnitFrequencyKernel} for $|\xi| \gg |t|$, we can localize the integral in \eqref{eq:PhillipsReduction} to $|\xi| \leq T$: Let $\chi_T \in C^\infty_c(B(0,2T))$ be a radially decreasing function with $\chi_T(\xi) = 1$ for $\xi \in B(0,T)$. Let $\eta_T = 1-\chi_T$. We estimate by integration by parts, the $L^2$-estimate from Proposition \ref{prop:L2Estimate}, and the resolvent estimate from Proposition \ref{prop:OperatorNorms}:
\begin{equation*}
\begin{split}
\big\| \int_{\R^d} \eta_T(\xi) K_t(\xi) e^{i \xi.D_a} u(x) d\xi \|_{L^\infty(\R^d)} &\lesssim \big\| \int_{\R^d} \eta_T(\xi) K_t(\xi) (1-\Delta_{\xi})^{2N} (1+L)^{-2N} e^{i \xi.D_a} u(x) d\xi \big\|_{L^\infty} \\
&\lesssim \big\| \int_{\R^d} (1-\Delta_{\xi})^N ( \eta_T(\xi) K_t(\xi)) e^{i \xi.D_a} (1+L)^{-N} u(x) d\xi \big\|_{L^2} \\
&\lesssim  \int_{|\xi| \geq T} (1+ |\xi|)^{-N} d \xi \|(1+L)^{-N} u \|_{L^2} \\
&\lesssim (1+T)^{-M} \| u \|_{L^1}.
\end{split}
\end{equation*}
Thus, it is enough to show the estimate
\begin{equation}
\label{eq:DispersiveEstimateFunctionalCalculusRedux}
\sup_{x \in \R^d} \big| \int_{|\xi| \leq T} K_t(\xi) e^{i \xi.D_a} u_0(x) d \xi \big| \lesssim (1+|t|)^{-\sigma(\ell)} \| u_0 \|_{L^1(\R^d)},
\end{equation}
which follows from
\begin{equation}
\label{eq:DispersiveEstimateII}
\sup_{x \in \R^d} \int_{|\xi| \leq T} \big| e^{i \xi.D_a} u_0(x)  \big| d\xi \lesssim \| u_0 \|_{L^1(\R^d)}.
\end{equation}

By Lemma \ref{lem:FiniteSpeedPropagation}, we can suppose that $\text{supp}(u_0) \subseteq B(x,CT)$ to show \eqref{eq:DispersiveEstimateII}. Furthermore, by commutativity of the Dirac operators under the radial frequency constraint, we can write
\begin{equation}
\label{eq:Iteration}
\int_{B(0,T)} \big| e^{i \xi.D_a} u_0(x) d\xi \big| \leq \int_{-T}^T d\xi_2 \ldots \int_{-T}^T d\xi_d \big( \int_{-T}^T d\xi_1 \big| e^{i \xi_1.D_{a_1}} \big( e^{i \xi'.D_{a}'} u_0(x_1,x') \big) \big| \big)
\end{equation}
with $\xi' = (\xi_2,\ldots,\xi_d)$ and $D_a' = (D_{a_2},\ldots,D_{a_d})$.
It suffices to prove the estimate for a one-parameter group as the estimate for the above expression then follows by iteration. Thus, the proof will be complete once we show the following estimate:
\begin{equation}
\label{eq:TransportPropertyDiracOperator}
 \int_{-T}^T |e^{it \tilde{D}_a} u_0(x)| dt \lesssim \| u_0 \|_{L^1(\R)}
\end{equation}
for $u_0 \in L^1(\R)$ and $\text{supp}(u_0) \subseteq B(x,CT)$. We write $u(t,x) = e^{it D_a} u_0(x)$ and with the notations from the proof of Proposition \ref{prop:TransportGroupPerturbation}, we find by change of basis
\begin{equation}
\label{eq:ChangeBasis}
\int_{-T}^T |u(t,x)| dt \lesssim \int_{-T}^T |v(t,x)| dt
\end{equation}
with
\begin{equation*}
v(t,x) = T_t v_0(x) + \int_0^t T_{t-s} Ev(s,x) ds, 
\end{equation*}
where $T_t = \exp \text{diag}(t (ab)^{1/2}(x) \partial_x, -t (ab)^{1/2}(x) \partial_x)$. We can write $T_t v_0(x) = v_0(\chi(t,x))$ with $\chi(x,\cdot)$ a $C^{1}$-diffeomorphism and
\begin{equation*}
\int_{\R} |T_t v_0(x)| dt \leq C(\lambda, \Lambda) \| v_0 \|_{L^1(\R;\C^2)}.
\end{equation*}
Hence, integration in time of $v(t,x)$ yields
\begin{equation}
\label{eq:TimeIntegration}
\int_{-T}^T |v(t,x)| dt \leq C(\lambda,\Lambda) \| v_0 \|_{L^1(\R;\C^2)} + \int_{-T}^T dt \int_{-T}^T ds |Ev(s,\chi(t-s,x))|.
\end{equation}
We estimate the second term by the diffeomorphism property of $\chi(x,\cdot-s)$, Fubini's theorem, and $L^1$-boundedness of $E$:
\begin{equation}
\label{eq:ErrorEstimateI}
\begin{split}
\int_{-T}^T dt \int_{-T}^T ds |Ev(s,\chi(x,t-s))| &\leq C(\lambda,\Lambda) \int_{-T}^T ds \| Ev(s) \|_{L^1} \\
&\leq C(\lambda,\Lambda) \int_{-T}^T ds \| E \|_{L^1 \to L^1} \| v(s) \|_{L^1_x} \\
&\leq C(\lambda,\Lambda) \| E \|_{L^1 \to L^1} \int_{-T}^T ds \int_{-CT}^{CT} dx' |v(s,x+x')| \\
&\leq C(\lambda,\Lambda)\| E \|_{L^1 \to L^1} CT \sup_x \| v(t,x) \|_{L_t^1([-T,T])}.
\end{split}
\end{equation}
We use the bound for $\| E \|_{L^1 \to L^1}$ from \eqref{eq:LpBoundE}: Let $T= \frac{1}{2 \| E \|_{L^1 \to L^1} \cdot C \cdot C(\lambda,\Lambda)}$ such that we can absorb \eqref{eq:ErrorEstimateI} into the right-hand side of \eqref{eq:TimeIntegration} to find
\begin{equation*}
\int_{-T}^T | v(t,x) | dt \lesssim \| v_0 \|_{L^1(\R^2;\C)}.
\end{equation*}
Conclusively, \eqref{eq:TransportPropertyDiracOperator} follows from \eqref{eq:ChangeBasis}, the previous estimate, and a final change of basis. The proof is complete.
\end{proof}

In the following we show the dispersive estimate for arbitrary finite times. We start with the following growth bound as consequence of Gr\o nwall's lemma:
\begin{lemma}
\label{lem:LongDispersiveEstimate}
Let $u(t,x) = e^{it D_a} u_0$. Then, we find the following estimate to hold:
\begin{equation*}
\| u(t) \|_{L^1} \leq e^{C(\| a_i \|_{\dot{C}^{0,1}},\lambda,\Lambda) t} \| u_0 \|_{L^1}.
\end{equation*}
\end{lemma}
\begin{proof}
Starting with the representation by Duhamel's formula after change of basis
\begin{equation*}
v(t,x) = T_t v_0(x) + \int_0^t T_{t-s} Ev(s,\chi(t-s,x)) ds,
\end{equation*}
we obtain
\begin{equation*}
\begin{split}
\int_{\R} |v(t,x)| dx &\leq \int_{\R} | v_0(\chi(t,x)) | dx + \int_0^t \int_{\R} |(Ev)(s,\chi(t-s,x))| dx ds \\
&\leq C(\lambda,\Lambda) \| u_0 \|_{L^1(\R)} + \int_0^t C(\lambda,\Lambda) \| E \|_{L^1 \to L^1} \| v(s) \|_{L^1} ds.
\end{split}
\end{equation*}
By $\| E \|_{L^1 \to L^1} \leq C(\lambda,\Lambda,\| a_i \|_{\dot{C}^{0,1}})$, the proof is concluded by applying Gr\o nwall's lemma and inverting the change of basis.
\end{proof}

In the following we show that the dispersive estimate remains true on arbitrary time intervals, but with a bad constant.
\begin{lemma}
\label{lem:GrowthBound}
Let $\ell \in \{1,2\}$ and $A = \max_i \| a_i \|_{\dot{C}^{0,1}}$. Then, we find the following estimate to hold:
\begin{equation*}
\| P_1 e^{it L^{\frac{\ell}{2}}} \|_{L^1 \to L^\infty} \leq C_1 e^{C_2 T A} (1+|t|)^{-\sigma(\ell)} \quad (0 < |t| \leq T).
\end{equation*}
\end{lemma}
\begin{proof}
As in the proof of Proposition \ref{prop:DispersiveEstimateUnitFrequencies} by the localization and decay of the half-wave or Schr\"odinger kernel at unit frequencies \eqref{eq:UnitFrequencyKernel}, it suffices to prove
\begin{equation*}
| \int_{|\xi| \leq T} e^{i \xi.D_a} u_0 | \leq C_1 e^{C_2 T A} \| u_0 \|_{L^1(\R^d)}.
\end{equation*}
As above, it suffices to prove the estimate for one Dirac operator by their commutativity and iteration.
Let $u(t,x) = e^{it \tilde{D}_a} u_0(x)$. We show
\begin{equation*}
\int_{-T}^T |u(t,x)| dt \leq C_1 e^{C_2 T A} \| u_0 \|_{L^1(\R)}.
\end{equation*}
Proposition \ref{prop:DispersiveEstimateUnitFrequencies} yields $T'(A,\lambda,\Lambda)$ such that
\begin{equation*}
\int_{0}^{T'} |u(t,x)| dt \lesssim \| u_0 \|_{L^1(\R)}.
\end{equation*}
We partition $[0,T]$ into $N=\lfloor T/T' \rfloor + 1$ intervals $I_k=[a_k,b_k]$ of length $T'$ such that $\int_{I'} |u(t,x)| dt \lesssim \| u(a_k) \|_{L^1}$. Hence, by Lemma \ref{lem:LongDispersiveEstimate} we find
\begin{equation*}
\int_{0}^T |u(t,x)| dt \leq \sum_{k=0}^{N+1} \int_{I_k} |u(t,x)| dt \lesssim \sum_{k=0}^{N+1} e^{kT' C} \| u_0 \|_{L^1} \lesssim e^{T C} \| u_0 \|_{L^1(\R)}.
\end{equation*}
\end{proof}

For the proof of the global results and extending the above to coefficients with bounded variation, we use the following result due to Beli--Ignat--Zuazua:
\begin{theorem}[{\cite[Theorem~1.1]{BeliIgnatZuazua2016}}]
\label{thm:GlobalDispersionWaveEquation}
For any $a \in BV(\R)$ satisfying \eqref{eq:Ellipticity} and $Var(\log(a)) < 2 \pi$ there exists a positive constant $C(Var(a),\lambda,\Lambda)$ such that the solution $u$ to
\begin{equation*}
\left\{ \begin{array}{cl}
\partial_{tt} u &= \partial_x (a(x) \partial_x) u, \quad (t,x) \in \R \times \R, \\
u(0) &= u_0 \in L^1(\R), \qquad \dot{u}(0) = 0
\end{array} \right.
\end{equation*}
satisfies
\begin{equation}
\label{eq:TransportPropertyWaveEquation}
\sup_{x \in \R} \int_{\R} |u(t,x)| dt \leq C(Var(a),\lambda,\Lambda) \| u_0 \|_{L^1(\R)}.
\end{equation}
\end{theorem}
\begin{remark}
We note that Beli \emph{et al.} showed the threshold for $Var(\log(a))$ to be sharp for \eqref{eq:TransportPropertyWaveEquation} to hold.
\end{remark}

We show the corresponding result to Proposition \ref{prop:DispersiveEstimateUnitFrequencies} for functions with locally bounded variation. To this end, let
\begin{equation*}
Var_T(a) = \sup \{ \int_{I} d|a'(x)| \, : \, I \subseteq \R \text{ interval}, \, |I| = T \}.
\end{equation*}
\begin{proposition}
\label{prop:DispersivePropertyLocallyBoundedVariation}
Let $a_i \in BV_{loc}(\R)$, $i=1,\ldots,d$ satisfy \eqref{eq:Ellipticity} such that there is $T>0$ with $Var_T(\log( a_i )) < 2 \pi$ and $a_{i}=1$ for $i=d+1,\ldots,2d$. Then, there is $\tilde{T} = \tilde{T}(T,\lambda,\Lambda)$ such that
\begin{equation*}
\| P_1 e^{it L^{\frac{\ell}{2}}} \|_{L^1 \to L^\infty} \lesssim (1+|t|)^{-\sigma(\ell)} \quad (0<|t| \leq \tilde{T}).
\end{equation*} 
If $Var(\log(a_i)) < 2 \pi$, then we can choose $\tilde{T} = \infty$.
\end{proposition}
\begin{proof}
By the reductions from above, it is enough to prove
\begin{equation*}
\int_{-\tilde{T}}^{\tilde{T}} | e^{i \xi.D_a} u_0(x)| d\xi \leq C(Var_T(a), \lambda, \Lambda) \|u_0 \|_{L^1(\R)}.
\end{equation*}
$\tilde{T}$ will be determined from $T$ and $\lambda$, $\Lambda$. It suffices to show the above display in one dimension:
\begin{equation*}
\int_{- \tilde{T}}^{\tilde{T}} | e^{it.D_a} u_0(x)| dt \leq C \| u_0 \|_{L^1(\R)}.
\end{equation*}
By the symmetry of the kernel, we change to the cosine group, and it is enough to prove
\begin{equation*}
\int_{- \tilde{T}}^{\tilde{T}} | \cos(t L) u_0(x)| dt \leq C \| u_0 \|_{L^1(\R)}
\end{equation*}
for $L = \sqrt{- \partial_x (a_i(x) \partial_x)}$. Note that the second entry on the diagonal can be recast into divergence form after change of variables as pointed out in \cite{BurqPlanchon2006}. By finite speed of propagation, we can localize $u_0$ to an interval of length $l = C(\lambda,\Lambda) \tilde{T}$ such that $\cos(tL) u_0(x) = \cos(t L) \tilde{u}_0(x)$ for $|t| \leq \tilde{T}$. Secondly, we can localize the coefficients of $L$ such that
\begin{equation*}
\cos(t \tilde{L}) \tilde{u}_0(x) = \cos(t L) u_0(x).
\end{equation*}
Note that this strictly speaking only works for piecewise constant coefficients, but we can assume this by approximation arguments (cf. \cite{BeliIgnatZuazua2016}). Moreover, the localization of $a$ can be chosen in an interval of length $C(\lambda,\Lambda) \tilde{T}$ around $x$. By hypothesis, for $C(\lambda,\Lambda) \tilde{T} < T$, we have by the localization and Theorem \ref{thm:GlobalDispersionWaveEquation}
\begin{equation*}
\int_{-\tilde{T}}^{\tilde{T}} | \cos(t L) u_0(x)| dt \leq \int_{\R} | \cos(t \tilde{L}) u_0(x) | dt \leq C(Var_T(a), \lambda, \Lambda) \| u_0 \|_{L^1(\R)}.
\end{equation*}
\end{proof}

\begin{remark}
The hypothesis on locally bounded variation is satisfied for $a_i \in C^1(\R)$ with $\| \partial_x a_i \|_{L^1(\R)} < \infty$ or $a_i \in C^1(\R)$ being $\tau$-periodic.
\end{remark}

We begin the proof of Theorem \ref{thm:StrichartzEstimate} in earnest:
\begin{proof}[Proof of Theorem \ref{thm:StrichartzEstimate}]
We have the scaling symmetry
\begin{equation*}
x \to \| a_i \|_{\dot{C}^{0,1}} x, \qquad t \to \| a_i \|_{\dot{C}^{0,1}}^{\ell} t,
\end{equation*}
with $\ell =1$ for the half-wave equation and $\ell=2$ for the Schr\"odinger equation, which reduces the estimate to
\begin{equation*}
\| |D|^{-s_{\ell}} e^{it L^{\frac{\ell}{2}}} u_0 \|_{L^p([0,T],L^q)} \lesssim T^{\frac{1}{p}} \| u_0 \|_{L^2}
\end{equation*}
with $\| a \|_{\dot{C}^{0,1}} \leq 1$. For a sharp Strichartz pair $(s,p,q,d)$ we always have $0 \leq s <1$. Hence, for Lipschitz coefficients the homogeneous Sobolev spaces $\dot{W}^{s,p}$ and $\dot{W}^{s,p}_L$ are equivalent for $-1 < s < 1$, and it suffices to estimate
\begin{equation*}
\| |D_L|^{-s_\ell} e^{it L^{\frac{\ell}{2}}} u_0 \|_{L^p([0,T],L^q)} \lesssim T^{\frac{1}{p}} \| u_0 \|_{L^2}.
\end{equation*}

The estimate for the low frequencies
\begin{equation*}
\begin{split}
\| |D_L|^{-s_\ell} e^{it L^{\frac{\ell}{2}}} \phi(|D_L|) u_0 \|_{L^p([0,T],L^q(\R^d))} &\lesssim T^{\frac{1}{p}} \| |D_L|^{\frac{\ell}{p}} e^{it L^{\frac{1}{2}}} \phi(|D_L|) u_0 \|_{L^\infty([0,T],L^2(\R^d))} \\
&\lesssim T^{\frac{1}{p}} \| u_0 \|_{L^2(\R^d)}
\end{split}
\end{equation*}
follows by H\"older's and Bernstein's inequality and boundedness of the group on $L^2(\R^d)$.

We turn to the estimate for the high frequencies: The square function estimate \eqref{eq:SquareFunctionEstimate} and Minkowski's inequality, recall that $p,q \geq 2$, yield
\begin{equation}
\label{eq:FrequencyLocalization}
\begin{split}
\| |D_L|^{-s_\ell} e^{it L^{\frac{\ell}{2}}} u_0 \|_{L^p([0,T],L^q(\R^d))} &\sim \| \big( \sum_{k \geq 1} | |D_L|^{-s_\ell} \psi_k(|D_L|) e^{it L^{\frac{\ell}{2}}} u_0 |^2 \big)^{\frac{1}{2}} \|_{L^p([0,T],L^q(\R^d))} \\
&\lesssim \big( \sum_{k \geq 1} 2^{-2ks_\ell} \| \psi_k(|D_L|) e^{it L^{\frac{\ell}{2}}} u_0 \|^2_{L^p([0,T],L^q(\R^d))} \big)^{\frac{1}{2}}.
\end{split}
\end{equation}
Hence, it suffices to show the frequency localized estimate
\begin{equation}
\label{eq:FrequencyLocalizationStrichartzReduction}
\| \psi_k(|D_L|) e^{it L^{\frac{\ell}{2}}} u_0 \|_{L^p([0,T],L^q(\R^d))} \lesssim T^{\frac{1}{p}} 2^{k s_\ell} 2^{\frac{k (\ell -1)}{p}} \| \psi_k(|D_L|) u_0 \|_{L^2(\R^d)}
\end{equation}
We use the scaling $x' = 2^k x$, $t' = 2^{\ell k} t$, which yields $\partial_{x} = 2^{k} \partial_{x'}$ and $\frac{D_a}{2^k} = D_{\tilde{a}}$, $\tilde{a}_i(x) = a_i(2^{-k} x_i)$. Note that
\begin{equation}
\label{eq:RescaledDerivatives}
\| a_i(2^{-k} \cdot) \|_{\dot{C}^{0,1}} \lesssim 2^{-k} \| a_i \|_{\dot{C}^{0,1}} \lesssim 2^{-k}. 
\end{equation}
This reduces \eqref{eq:FrequencyLocalizationStrichartzReduction} to unit frequencies:
\begin{equation}
\label{eq:RescaledtoUnitFrequencies}
\| \psi(|D_{\tilde{L}}|) e^{it \tilde{L}^{\frac{\ell}{2}}} u_0 \|_{L^p([0,T 2^{\ell k}], L^q(\R^d))} \lesssim 
\big( T 2^{(\ell-1)k} \big)^{\frac{1}{p}} \| u_0 \|_{L^2(\R^d)}.
\end{equation}
We note that as a consequence of \eqref{eq:RescaledDerivatives} by Proposition \ref{prop:DispersiveEstimateUnitFrequencies} there is $C$ independent of $k \in \mathbb{N}_0$ such that the dispersive estimate
\begin{equation}
\label{eq:DispersiveEstimateLongTime}
\| \psi(|D_{\tilde{L}}|) e^{it \tilde{L}^{\frac{\ell}{2}}} u_0 \|_{L^\infty(\R^d)} \lesssim (1+|t|)^{-\sigma(\ell)} \| u_0 \|_{L^1(\R^d)}
\end{equation}
holds for $0 \leq |t| \leq C 2^k$ with implicit constant independent of $k$ and $t$. Thus, by the energy estimate from Proposition \ref{prop:L2Estimate}, the dispersive estimate \eqref{eq:DispersiveEstimateLongTime}, and Theorem \ref{thm:KeelTao}, we find
\begin{equation}
\label{eq:StrichartzEstimateLongTime}
\| \psi(|D_{\tilde{L}}|) e^{it \tilde{L}^{\frac{\ell}{2}}} u_0 \|_{L^p([0,C 2^k],L^q(\R^d))} \lesssim \| u_0 \|_{L^2(\R^d)}.
\end{equation}
Hence, by partitioning $[0,T2^{\ell k}]$ into $O(T 2^{(\ell-1) k})$ intervals $(I_m)$ of length $C2^{k}$, we conclude
\begin{equation*}
\begin{split}
&\quad \| \psi(|D_{\tilde{L}}|) e^{it \tilde{L}^{\frac{\ell}{2}}} u_0 \|_{L^p([0,T2^{\ell k}],L^q(\R^d))} \\
&= \big( \sum_m \| \psi(|D_{\tilde{L}}|) e^{it \tilde{L}^{\frac{\ell}{2}}} u_0 \|^p_{L^p(I_m,L^q(\R^d))} \big)^{\frac{1}{p}} \lesssim \big( \# I_m \big)^{\frac{1}{p}} \| u_0 \|_{L^2(\R^d)}.
\end{split}
\end{equation*}
We have additionally used that $e^{it \tilde{L}^{\frac{\ell}{2}}}$ conserves the $L^2$-norm independently of $k$.


Thus, by change of variables, we can conclude the Strichartz estimates for high frequencies of the original operators from estimates for unit frequencies of the rescaled operator:
\begin{equation*}
\begin{split}
\| e^{it L^{\frac{\ell}{2}}} \psi_k(|D_L|) u_0  \|_{L^p([0,T],L^q )} &\lesssim 2^{- \frac{\ell k}{p}} 2^{- \frac{dk}{q}} \| e^{it \tilde{L}^{\frac{\ell}{2}}} \psi(|D_{\tilde{L}}|) \tilde{u}_0 \|_{L^p([0,T 2^{\ell k}],L^q)} \\
&\lesssim (T 2^{(\ell -1)k})^{\frac{1}{p}} 2^{-\frac{\ell k}{p}} 2^{- \frac{dk}{q}} \| \psi(|D_{\tilde{L}}|) \tilde{u}_0 \|_{L^2} \\
&\lesssim (T 2^{(\ell -1)k})^{\frac{1}{p}} 2^{- \frac{\ell k}{p}} 2^{- \frac{dk}{q}} 2^{\frac{dk}{2}} \| \psi_k(|D_L|) u_0 \|_{L^2} \\
&= (T 2^{(\ell -1)k})^{\frac{1}{p}} 2^{ks_{\ell}} \| \psi_k(|D_L|) u_0 \|_{L^2}.
\end{split}
\end{equation*}
The proof is concluded by \eqref{eq:FrequencyLocalization} and square summing the above display.
\end{proof}
Next, we turn to the proof of global-in-time Strichartz estimates. The additional hypothesis on small variation allows for dispersive estimates with uniform constant for arbitrary times, from which global Strichartz estimates follow.
\begin{proposition}
\label{prop:GlobalDispersiveEstimate}
Let $a_i \in C^{0,1}(\R)$ satisfy \eqref{eq:Ellipticity} and suppose that $Var(\log(a_i )) < 2 \pi$ for $i=1,\ldots,d$ and $a_{i} =1$ for $i=d+1,\ldots,2d$. Then, we find the following estimates to hold:
\begin{align}
\label{eq:GlobalDispersiveEstimateWave}
\| e^{it L^{\frac{1}{2}}} \psi(|D_L|) \|_{L^1 \to L^\infty} &\lesssim (1+|t|)^{- \frac{d-1}{2}}, \\
\label{eq:GlobalDispersiveEstimateSEQ}
\| e^{itL} \|_{L^1 \to L^\infty} &\lesssim |t|^{-\frac{d}{2}} \quad (t \neq 0)
\end{align}
with implicit constant only depending on the ellipticity constants and $Var(\log(a_i a_{i+d}))$.
\end{proposition}
\begin{proof}
As in the proof of Proposition \ref{prop:DispersiveEstimateUnitFrequencies}, by the decay of the half-wave kernel at unit frequencies, to show \eqref{eq:GlobalDispersiveEstimateWave}, it suffices that
\begin{equation}
\label{eq:GlobalTransportProperty}
\int_{\R^d} \big| e^{i \xi.D_a} f(x) \big| d\xi \lesssim \int_{\R^d} |f(x)| dx.
\end{equation}
For the proof of \eqref{eq:GlobalDispersiveEstimateSEQ}, we write by Phillips functional calculus without frequency localization
\begin{equation*}
e^{it L} u_0(x) = \int_{\R^d} G_t(\xi) e^{i \xi.D_a} u_0(x) d\xi
\end{equation*}
with
\begin{equation*}
G_t(\xi) = \frac{1}{(2 \pi)^d} \int_{\R^d} e^{i y.\xi} e^{it |y|^2} dy = \frac{1}{(4 \pi i t)^{\frac{d}{2}}} e^{it \xi^2}.
\end{equation*}
Thus,
\begin{equation*}
|e^{it L} u_0(x) | \lesssim \sup_{\xi \in \R^d} |G_t(\xi) | \int_{\R^d} | e^{i \xi.D_a} u_0(x)| d\xi \lesssim |t|^{-\frac{d}{2}} \int_{\R^d} |e^{i \xi.D_a} u_0(x) | d\xi,
\end{equation*}
and \eqref{eq:GlobalDispersiveEstimateSEQ} follows likewise from \eqref{eq:GlobalTransportProperty}.

As in the proof of Proposition \ref{prop:DispersiveEstimateUnitFrequencies}, by commutativity of the generators and iteration, it suffices to show
\begin{equation}
\label{eq:OneDimensionalTransportProperty}
\int_{\R} \big| e^{it \tilde{D}_a} u(x) \big| dt \lesssim \int_{\R} |u(x)| dx
\end{equation}
for $u \in C^\infty_c(\R)$. We write $u(t,x) = e^{it \tilde{D}_a} u(x)$. Since we can change to the cosine group by radial frequency constraint, \eqref{eq:OneDimensionalTransportProperty} is immediate from Theorem \ref{thm:GlobalDispersionWaveEquation}.
\end{proof}
With global dispersive estimates at hand, we turn to the proof of Theorem \ref{thm:GlobalEstimates}:
\begin{proof}[Proof of Theorem \ref{thm:GlobalEstimates}]
We begin with the proof of \eqref{eq:GlobalSEQStrichartz}, which does not require an additional Littlewood-Paley decomposition. We have to show
\begin{equation*}
\| e^{itL} u_0 \|_{L^p(\R; L^q(\R^d))} \lesssim \| u_0 \|_{L^2(\R^d)}
\end{equation*}
for $\frac{2}{p} + \frac{d}{q} = \frac{d}{2}$, $(p,q,d) \neq (2,\infty,2)$. This is a consequence of Theorem \ref{thm:KeelTao} due to the global energy and dispersive estimates from Proposition \ref{prop:L2Estimate} and \ref{prop:GlobalDispersiveEstimate}.

For the proof of \eqref{eq:GlobalWaveStrichartz}, it suffices again to consider only sharp Strichartz pairs. We use the square function estimate and Minkowski's inequality to find
\begin{equation}
\label{eq:GlobalStrichartzFrequencyLocalization}
\begin{split}
\| |D|^{-s_1} e^{it L^{\frac{1}{2}}} u_0 \|_{L^p(\R;L^q(\R^d))} &\sim \| \big( \sum_{k \in \Z} | |D_L|^{-s_1} \psi_k(|D_L|) e^{it L^{\frac{1}{2}}} u_0 |^2 \big)^{\frac{1}{2}} \|_{L^p(\R;L^q(\R^d))} \\
&\lesssim \big( \sum_{k \in \Z} 2^{-2s_1k} \| \psi_k(|D_L|) e^{it L^{\frac{1}{2}}} u_0 \|^2_{L^p(\R;L^q(\R^d))} \big)^{\frac{1}{2}}.
\end{split}
\end{equation}
We rescale similarly as in the proof of Theorem \ref{thm:StrichartzEstimate} to reduce to unit frequencies. It suffices to prove that
\begin{equation}
\label{eq:GlobalUnitFrequencyStrichartzEstimate}
\| \psi(|D_{\tilde{L}}|) e^{it \tilde{L}^{\frac{1}{2}}} \tilde{u}_0 \|_{L^p(\R;L^q(\R^d))} \lesssim \| \tilde{u}_0 \|_{L^2(\R^d)}
\end{equation}
with implicit constant uniform in the rescalings. This is true as the rescaled coefficients are given by $a_{i,k}(x) = a_i(2^{-k} x)$ and hence, $\| \partial_x a_{i,k} \|_{L^1(\R)} = \| \partial_x a_i \|_{L^1(\R)}$. Thus, the global dispersive estimate derived in Proposition \ref{prop:GlobalDispersiveEstimate} holds uniformly in $k$, so does the energy estimate by Proposition \ref{prop:L2Estimate}, and hence \eqref{eq:GlobalUnitFrequencyStrichartzEstimate} holds true independently of $k \in \Z$. Plugging \eqref{eq:GlobalUnitFrequencyStrichartzEstimate} into \eqref{eq:GlobalStrichartzFrequencyLocalization} and square summing over the spectrally localized pieces finishes the proof.
\end{proof}

Next, we discuss inhomogeneous estimates. The involved arguments are standard by now, so we shall be brief. The most straight-forward estimate is recorded in Corollary \ref{cor:StrichartzEnergyEstimate}. By local well-posedness of the half-wave and Schr\"odinger equation in $L^2$, we can make use of Duhamel's formula:
\begin{equation*}
u = e^{it L^{\frac{\ell}{2}}} u_0 + \int_0^t e^{i(t-s)L^{\frac{\ell}{2}}} (P_{\ell} u)(s) ds
\end{equation*}
with $P_{\ell} = i \partial_t + L^{\frac{\ell}{2}}$. Corollary \ref{cor:StrichartzEnergyEstimate} now follows from Minkowski's inequality and homogeneous estimates. For details we refer to \cite[Corollary~2.10]{BurqGerardTzvetkov2004}.

For the proof of Corollary \ref{cor:InhomogeneousChristKiselev}, we use the following taylored version of the Christ--Kiselev lemma \cite{ChristKiselev2001}:
\begin{lemma}[{\cite[Lemma~8.1]{HassellTaoWunsch2006}}]
\label{lem:ChristKiselev}
Let $X$ and $Y$ be Banach spaces and for all $s,t \in \R$ let $K(s,t): X \to Y$ be an operator-valued kernel from $X$ to $Y$. Suppose we have the estimate
\begin{equation*}
\| \int_{\R} K(s,t) f(s) ds \|_{L^q(\R,Y)} \leq A \| f \|_{L^p(\R,X)}
\end{equation*}
for some $A > 0$ and $1 \leq p < q \leq \infty$, and $f \in L^p(\R;X)$. Then, we have

\begin{equation*}
\| \int_{s<t} K(s,t) f(s) ds \|_{L^q(\R,Y)} \leq C_{p,q} A \| f \|_{L^p(\R,X)}.
\end{equation*}
\end{lemma}

Theorem \ref{thm:GlobalInhomogeneousStrichartzEstimates} is another consequence of Theorem \ref{thm:KeelTao} with the dispersive estimate at hand. Additionally, a Littlewood-Paley decomposition is required. 

\section{Applications}
\label{section:Applications}
In this section we give applications of the preceding analysis. First, we show Strichartz estimates for coefficients of bounded variation by limiting arguments. Next, we prove wave Strichartz estimates for H\"older coefficients under the additional structural assumptions. In this case, they improve the general estimates due to Tataru \cite{Tataru2001}. Finally, we indicate applications to the well-posedness theory of nonlinear Schr\"odinger equations.

\subsection{Strichartz estimates for BV coefficients}
\label{subsection:BVCoefficients}
In the following we see that by approximation arguments due to Burq--Planchon \cite[Proposition~1.3]{BurqPlanchon2006}, we can drop the hypothesis that the coefficients are Lipschitz and allow for $BV_{loc}(T)$-coefficients, which norm is defined by
\begin{equation*}
\| a \|_{BV_{loc}(T)} = \sup \{ \int_I d|a'|(y) \, : \, I \subseteq \R \text{ interval, }  |I| = T \}
\end{equation*}
for $a'$ being a locally finite measure. The backbone is the dispersive estimate only depending on the ellipticity constants and $\| a \|_{BV_{loc}(T)}$.
We have the following:
\begin{theorem}
\label{thm:BVVariant}
Let $(a_i)_{i=1,\ldots,d} \subseteq L^\infty$ satisfy \eqref{eq:Ellipticity} and let $T>0$ such that $\| \log(a_i) \|_{BV_{loc}(T)} < 2 \pi$ and $L = - \sum_{i=1}^d \partial_{x_i} ( a_i(x_i) \partial_{x_i})$. Then, we find the following estimates to hold:
\begin{equation*}
\|  e^{it L^{\frac{\ell}{2}}} u_0 \|_{L^p([0,T], \dot{B}^{-s_\ell,q}_2(\R^d) )} \lesssim \| u_0 \|_{H^{\frac{\ell-1}{p}}(\R^d)}
\end{equation*}
for $\ell \in \{1,2\}$, $s_{\ell} = d \big( \frac{1}{2} - \frac{1}{q} \big) - \frac{\ell}{p}$, $(p,q,d)$ being wave $(\ell =1)$ or Schr\"odinger $(\ell = 2)$ admissible.
\end{theorem}

 A straight-forward modification of the proof of \cite[Proposition~1.3]{BurqPlanchon2006} yields the following:
\begin{proposition}
Let $(a_i)^{d}_{i=1} \subseteq C^{0,1}(\R;\R)$ satisfy \eqref{eq:Ellipticity} and $\| a_i \|_{BV_{loc}(T)} < \infty$.\\
Denote $L = - \sum_{i=1}^d \partial_{x_i} (a(x_i) \partial_{x_i})$. Suppose that the $C_0$-group $S_a(t)$ generated by $iL$ satisfies
\begin{equation*}
\| S_a(t) u_0 \|_B \leq C(\lambda,\Lambda, \| a_i \|_{BV_{loc}(T)}) \| u_0 \|_{L^2(\R^d)}
\end{equation*}
with $B$ a Banach space (weakly) continuously embedded in $\mathcal{D}'(\R^{d+1})$, whose unit ball is weakly compact. Then the same result holds for $(a_i)_{i=1}^d \subseteq L^\infty(\R)$ satisfying \eqref{eq:Ellipticity} and $ \| a_i \|_{BV_{loc}(T)} < \infty $.
\end{proposition}
This yields the Schr\"odinger Strichartz estimates of Theorem \ref{thm:BVVariant}.
We turn to Strichartz estimates for wave equations locally-in-time with coefficients $\| a_i \|_{BV_{loc}(T)} < \infty$:
\begin{proposition}
\label{prop:LocalWaveStrichartzBV}
Let $d \geq 2$ and $(a_i)_{i=1}^d \subseteq C^{0,1}(\R;\R)$ satisfy \eqref{eq:Ellipticity}. Denote $L = - \sum_{i=1}^d \partial_{x_i} (a_i(x_i) \partial_{x_i})$. Suppose that the $C_0$-group $S_a(t)$ generated by
\begin{equation*}
\begin{pmatrix}
0 & 1 \\ 
L & 0
\end{pmatrix}
\text{ on } H^1 \times L^2
\end{equation*}
satisfies
\begin{equation*}
\| (S_a(t)(u_0, v_0))_1 \|_B \leq C(\lambda,\Lambda,\| a_i \|_{BV_{loc}(T)}) \| (u_0, \partial_t u_0) \|_{H^1 \times L^2} 
\end{equation*}
with $B$ a Banach space (weakly) continuously embedded in $\mathcal{D}'([0,T] \times \R^d)$, whose unit ball is weakly compact. Then the same result holds for $(a_i)_{i=1}^d \subseteq BV_{loc}(T)$ satisfying \eqref{eq:Ellipticity}.
\end{proposition}
\begin{proof}
Let $(a_i)_{i=1}^d \subseteq BV_{loc}(T)$ satisfy \eqref{eq:Ellipticity}. Let $(\rho_\varepsilon)_{\varepsilon > 0}$ be a family of positive mollifiers and consider $a_i^\varepsilon = \rho_\varepsilon * a_i$. We have $(a_i^\varepsilon) \subseteq  C^\infty$ and $\| \partial_x (a_i^\varepsilon) \|_{L^1_{loc}(T)} \leq \| a_i \|_{BV_{loc}(T)}$. Hence, by assumption
\begin{equation*}
\| (S_{a_\varepsilon}(t)(u_0, v_0))_1 \|_{B} \leq C(\lambda,\Lambda,\| a_i \|_{BV_{loc}(T)}) \| (u_0,v_0) \|_{H^1 \times L^2}
\end{equation*}
and $(S_{a_\varepsilon}(t)(u_0,v_0))_1$ converges weakly in $\mathcal{D}'([0,T] \times \R^d)$. As in \cite{BurqPlanchon2006}, it suffices to prove that $S_{a_\varepsilon}(t) (u_0,v_0) \to S_a(t) (u_0, v_0)$ in $H^1 \times L^2$ because by energy estimates this yields convergence to $S_a(\cdot) (u_0,v_0)$ in $\mathcal{D}'([0,T] \times \R^d)$. For this purpose it is enough to show the strong convergence of 
\begin{equation*}
\mathcal{L}_\varepsilon^{-1} = \begin{pmatrix}
i & 1 \\
L_\varepsilon & i
\end{pmatrix}
^{-1} \text{ to } \mathcal{L}^{-1} = 
\begin{pmatrix}
i & 1 \\
L & i
\end{pmatrix}^{-1}
\end{equation*}
 by \cite[Theorem~VIII.21]{ReedSimon1980}. Clearly, $\mathcal{L}_\varepsilon^{-1}:L^2 \times H^{-1} \to H^1 \times L^2$ uniformly in $\varepsilon$ as well as $\mathcal{L}^{-1}: L^2 \times H^{-1} \to H^1 \times L^2$. By the resolvent formula,
 \begin{equation*}
 \mathcal{L}^{-1} - \mathcal{L}_\varepsilon^{-1} = \mathcal{L}_\varepsilon^{-1} ( \mathcal{L}_\varepsilon - \mathcal{L}) \mathcal{L}^{-1}.
 \end{equation*}
 Hence, $\mathcal{L}_\varepsilon^{-1}$ converges strongly to $\mathcal{L}^{-1}$ as operators from $L^2 \times H^{-1}$ to $H^1 \times L^2$. Hence, $\mathcal{L}_\varepsilon^{-1}$ converges to $\mathcal{L}^{-1}$ strongly as operators from $H^1 \times L^2$ to $H^1 \times L^2$. The proof is complete.
\end{proof}
Now we can prove the Strichartz estimates for half-wave equations stated in Theorem \ref{thm:BVVariant}.
\begin{proof}[Proof of Theorem \ref{thm:BVVariant}, Half-wave Strichartz estimates]
The half-wave Strichartz estimates for Lipschitz coefficients read
\begin{equation*}
\| e^{it L^{\frac{1}{2}}} u_0 \|_{L_T^p \dot{B}^{-s_1,q}_2} \lesssim \| u_0 \|_{L^2(\R^d)}.
\end{equation*}
These yield wave Strichartz estimates for $C^{0,1}$-coefficients:
\begin{equation*}
\| |D|^{1-s} (S_a(t) (u_0,v_0))_1 \|_{L^p_T L^q} \lesssim \| u_0 \|_{\dot{H}^1} + \| v_0 \|_{L^2}.
\end{equation*}
We have $(S_a(t) (u_0,v_0))_1 = \cos(tL) u_0 + \frac{\sin(tL) v_0}{L^{\frac{1}{2}}}$. Since it is enough to consider sharp Strichartz pairs with $q>2$, we have $0 < 1-s <1$ and have equivalence of Besov norms by Proposition \ref{prop:EquivalenceBesovNorms}. We estimate
\begin{equation*}
\begin{split}
\| (S_a(t) (u_0,v_0))_1 \|_{L^p_T \dot{B}^{1-s,q}_2} &\lesssim \| \cos(tL) u_0 \|_{L_T^p \dot{B}^{1-s,q}_{2,L}} + \| \sin(tL) v_0 \|_{L_T^p \dot{B}^{-s,q}_{2,L}} \\
&\lesssim \| |D_L| u_0 \|_{L^2} + \| v_0 \|_{L^2} \lesssim \| u_0 \|_{\dot{H}^1} + \| v_0 \|_{L^2}.
\end{split}
\end{equation*}
At this point, we invoke Proposition \ref{prop:LocalWaveStrichartzBV}, which yields estimates
\begin{equation*}
\| (S_a(t)(u_0,v_0))_1 \|_{L_T^p \dot{B}^{1-s,q}_2} \lesssim \| u_0 \|_{\dot{H}^1} + \| v_0 \|_{L^2}
\end{equation*}
for $BV_{loc}(T)$-coefficients. We apply this estimate to $v_0 = i L^{\frac{1}{2}} u_0$, which gives
\begin{equation*}
\| e^{it L^{\frac{1}{2}}} u_0 \|_{L_T^p \dot{B}^{1-s,q}_{2,L}} \lesssim \| u_0 \|_{\dot{H}^1}.
\end{equation*}
The half-wave Strichartz estimate now follows from trading derivatives and the substitution $v_0 = |D_L| u_0$.
\end{proof}
As an example for new local-in-time Strichartz estimates for Schr\"odinger equations with $BV_{loc}$-coefficients, we consider the Kronig--Penney model: Let $x_0 \in (0,1)$, and $b_0 \neq b_1 > 0$ with $b_0 x_0 = b_1 (1-x_0)$. Consider the $1$-periodic function $a:\R \to \R$ defined by
\begin{equation*}
a(x) =
\begin{cases}
b_0^{-2} \text{ for } x \in [0,x_0), \\
b_1^{-2} \text{ for } x \in [x_0,1).
\end{cases}
\end{equation*}
Banica \cite[Theorem~1.2]{Banica2003} showed that the dispersive estimate
\begin{equation*}
\| u(t) \|_{L^\infty(\R)} \lesssim |t|^{-\frac{1}{2}} \| u_0 \|_{L^1(\R)} \quad (t \neq 0)
\end{equation*}
fails for solutions
\begin{equation*}
\left\{ \begin{array}{cl}
i \partial_t u + \partial_x (a(x) \partial_x) u &= 0 \quad (t,x) \in \R \times \R, \\
u(0) &= u_0.
\end{array} \right.
\end{equation*}
Still we find the Strichartz estimates with derivative loss
\begin{equation}
\label{eq:StrichartzKroenigPenneyModel}
\| u \|_{L^p([0,1],L^q(\R))} \lesssim \| u_0 \|_{H^{\frac{1}{p}}(\R)}
\end{equation}
for $\frac{2}{p} + \frac{1}{q} \leq \frac{1}{2}$, $p,q \geq 2$ to hold, which indicates dispersive properties on frequency dependent time scales
\begin{equation*}
\| S^A_N u(t) \|_{L^\infty(\R)} \lesssim |t|^{-\frac{1}{2}} \| S^A_N u_0 \|_{L^1(\R)}
\end{equation*}
for $0<|t| \lesssim N^{-1}$ with $S^A_N$ denoting the spectral projection to $[N,2N)$ of $-\partial_x (a(\cdot) \partial_x)$. This is reminiscent of the short-time dispersive estimates on smooth compact manifolds due to Burq--G\'erard--Tzvetkov \cite{BurqGerardTzvetkov2004}. 

\subsection{Strichartz estimates for H\"older coefficients}
\label{subsection:HoelderCoefficients}

To show Strichartz estimates for H\"older coefficients, we firstly derive Strichartz estimates with inhomogeneity in $L^1 L^2$ and with precise dependence on the Lipschitz norm and time interval. We show the following refinement of Corollary \ref{cor:StrichartzEnergyEstimate}:
\begin{proposition}
\label{prop:MuDependence}
Let $(\rho,p,q,d)$ a wave admissible Strichartz pair. Then, we find the following estimate to hold:
\begin{equation}
\label{eq:RefinedL1L2Estimate}
\| |D|^{-\rho} u \|_{L^p([0,T],L^q)} \lesssim \mu^{\frac{1}{p}} \| u \|_{L^\infty L^2(\R^d)} + \mu^{-\frac{1}{p'}} \| P_{1} u \|_{L^1([0,T], L^2)}
\end{equation}
with $\mu = T \| a \|^\ell_{\dot{C}^{0,1}} \geq 1$ and $P_1 = i \partial_t + L^{\frac{1}{2}}$.
\end{proposition}
\begin{proof}
We find by Minkowski's inequality and the homogeneous Strichartz estimate for admissible exponents
\begin{equation*}
\begin{split}
\| |D|^{-\rho} u \|_{L^p([0,T],L^q)} &\lesssim \mu^{\frac{1}{p}} \| u_0 \|_{L^2} + \int_0^T \mu^{\frac{1}{p}} \| P_1 u(s) \|_{L^2} ds \\
&\lesssim \mu^{\frac{1}{p}} \| u \|_{L^\infty L^2} + \mu^{\frac{1}{p}} \| P_{1} u \|_{L^1 L^2}.
\end{split}
\end{equation*}
We shall see that we can improve the constants: Suppose that $\| a \|_{\dot{C}^{0,1}} \lesssim 1$ (by rescaling). Divide $[0,T]$ into intervals $I_j= [t_j, t_{j+1}]$ such that 
\begin{equation*}
\| P_{1} u \|_{L^1([t_j,t_{j+1}],L^2)} \leq T^{-1} \| P_{1} u \|_{L^1([0,T],L^2)}
\end{equation*}
and $t_{j+1} - t_j \leq 1$. Then, there are roughly $T$ intervals $I_j$ and for each interval $I_j$, we find by the above argument
\begin{equation*}
\| |D|^{-\rho} u \|_{L^p(I_j,L^q)} \lesssim \| u(t_j) \|_{L^2} + T^{-1} \| P_{1} u \|_{L^1([0,T],L^2)}.
\end{equation*}
Taking the $\ell^p$-sum over intervals $I_j$, we find
\begin{equation*}
\| |D|^{-\rho} u \|_{L^p L^q} \lesssim T^{\frac{1}{p}} \| u \|_{L^\infty L^2} + T^{- \frac{1}{p'}} \| P_{1} u \|_{L^1([0,T],L^2)}.
\end{equation*}
For arbitrary $\| a \|_{\dot{C}^{0,1}}$-norm, we find
\begin{equation*}
\| |D|^{-\rho} u \|_{L^p([0,T],L^q)} \lesssim \mu^{\frac{1}{p}} \| u_0 \|_{L^2} + \mu^{-\frac{1}{p'}} \| P_1 u \|_{L^1 L^2}.
\end{equation*}
\end{proof}
In the following we derive Strichartz estimates for coefficients with lower regularity. The estimates are supposed to be understood as a priori estimates for smooth solutions to equations with smooth coefficients (but only depending on the rough norms). Later it becomes useful that for smooth solutions to wave equations
\begin{equation*}
\partial_t^2 u = \sum_{i=1}^d \partial_{x_i} (a_i(x_i) \partial_{x_i} u), \quad (t,x) \in \R \times \R^d,
\end{equation*}
the energy
\begin{equation}
\label{eq:Energy}
E_a(u) = \int_{\R^d} |\partial_t u|^2 + \sum_{i=1}^d a_i(x_i) (\partial_{x_i} u)^2
\end{equation}
is conserved. The arguments are adapted from Tataru's works \cite{Tataru2001,Tataru2002} to our setting. Dealing with time-independent coefficients simplifies matters. The idea remains the same: After truncating the coefficients in frequency, we arrive at Lipschitz coefficients with large Lipschitz norm. The argument only works for operators in divergence form. Let $P = \partial_t^2 - \sum_{i=1}^d \partial_i ( a_i(x_i) \partial_i )$ and let $(S_N)_{N \in 2^{\mathbb{N}_0}}$ denote inhomogeneous Littlewood-Paley projectors. We record the following consequence of Proposition \ref{prop:MuDependence}:
\begin{corollary}
\label{cor:WaveEstimateDivergence}
Let $d \geq 2$, $T>0$, and $a_i \in C^{0,1}(\R)$, $i=1,\ldots,d$ satisfy \eqref{eq:Ellipticity}. Suppose that
 $T \| a_i \|_{\dot{C}^{0,1}} \leq \mu$ for some $\mu \geq 1$. Then, we find the following estimate to hold:
\begin{equation}
\label{eq:WaveEquationEllipticOperator}
\| |D|^{1-\rho} u \|_{L^p([0,T],L^q(\R^d))} \lesssim \mu^{\frac{1}{p}} \| \nabla u \|_{L^\infty L^2} + \mu^{-\frac{1}{p'}} \| P u \|_{L^1([0,T],L^2(\R^d)}
\end{equation} 
 provided that $(\rho,p,q,d)$ is a sharp wave Strichartz pair.
\end{corollary}
We turn to the proof of Theorem \ref{thm:HolderCoefficients}:

\begin{proof}[Proof of Theorem \ref{thm:HolderCoefficients}]
By scaling invariance, we can suppose that $T=1$ and $\| a_i \|_{\dot{B}^s_{\infty,1}} \leq \mu$. Let $a_{i,\leq N}$ denote the coefficients with Fourier transform smoothly truncated at frequencies $\frac{N}{10}$, and
$P_N = \partial_t^2 - \sum_{i=1}^d \partial_i (a_{i,\leq N} \partial_i)$. We first argue that it is enough to show
\begin{equation}
\label{eq:DyadicEstimate}
N^{1-\rho - \frac{\sigma}{p}} \| S_N u \|_{L^p L^q} \lesssim \mu^{\frac{1}{p}} \| \nabla S_N u \|_{L^\infty L^2} + \mu^{-\frac{1}{p'}} \| |D|^{-\sigma} P_N S_N u \|_{L^1 L^2}.
\end{equation}
Since the spatial frequencies of $P_N S_N u$ are comparable to $N$, \eqref{eq:DyadicEstimate} is equivalent to
\begin{equation}
\label{eq:DyadicEstimateII}
N^{1-\rho} \| S_N u \|_{L^p L^q} \lesssim (N^\sigma \mu)^{\frac{1}{p}} \| \nabla S_N u \|_{L^\infty L^2} + (\mu N^\sigma)^{-\frac{1}{p'}} \| P_N S_N u \|_{L^1 L^2}.
\end{equation}
After observing that $\| a_{i, \leq N} \|_{\dot{C}^{0,1}} \lesssim \mu N^\sigma$,  \eqref{eq:DyadicEstimateII} follows from Corollary \ref{cor:WaveEstimateDivergence}. To see that \eqref{eq:DyadicEstimate} implies \eqref{eq:DyadicEstimateHolderCoefficients}, it suffices to see that
\begin{equation*}
\| |D|^{-\sigma} (S_N P - P_N S_N) u \|_{L^1 L^2} \lesssim \mu \| \nabla u \|_{L^\infty L^2}.
\end{equation*}
Let $v = \nabla_x u$, which is compactly supported in time, and since $P$ is in divergence form, it suffices to show the fixed-time estimate
\begin{equation*}
\| |D|^{1-\sigma} (S_N a_i - a_{i, \leq N} S_N) v(t) \|_{L^2} \lesssim \mu \| v(t) \|_{L^2}
\end{equation*}
or
\begin{equation*}
\| (S_N a_i - a_{i, \leq N} S_N) v(t) \|_{L^2} \lesssim N^{-s} \mu \| v(t) \|_{L^2}.
\end{equation*}
We prove this by considering dyadic blocks $S_K a_i = a_{i,K}$. We have $K^s \| a_{i,K} \|_{L^\infty} \lesssim \mu$. For $K \gg \frac{N}{10}$, the estimate reads
\begin{equation*}
\| S_N (a_{i,K} v(t)) \|_{L^2} \lesssim N^{-s} K^s \| a_{i,K} \|_{L^\infty} \| v(t) \|_{L^2},
\end{equation*}
which is immediate. For $K \lesssim \frac{N}{10}$, we have to prove
\begin{equation*}
\| (S_N a_{i,K} - a_{i,K} S_N) v(t) \|_{L^2} \lesssim N^{-s} K^s \| a_{i,K} \|_{L^\infty} \| v(t) \|_{L^2}.
\end{equation*}
By rescaling $K \to 1$, $N \to \frac{N}{K}$, it suffices to prove
\begin{equation*}
\| [S_N, a_{i,1}] v(t) \|_{L^2} \lesssim N^{-s} \| a_{i,1} \|_{L^\infty} \| v(t) \|_{L^2},
\end{equation*}
or, equivalently,
\begin{equation*}
\| [S_N, a_{i,1}] v(t) \|_{L^2} \lesssim N^{-s} \| \nabla_x a_{i,1} \|_{L^\infty} \| v(t) \|_{L^2}.
\end{equation*}
This is a well-known commutator estimate as the kernel is given by
\begin{equation*}
|K(x,y)| =  |k_N(x,y) (a_{i,1}(x) - a_{i,1}(y))| \lesssim_M (1+N|x-y|)^{-M} |x-y| \| \nabla_x a_{i,1} \|_{L^\infty}.
\end{equation*}
The proof is complete.
\end{proof}

Note that there is slack in the proof as the commutator estimate gives
\begin{equation*}
\| [S_N, a_{i,1}] v(t) \|_{L^2} \lesssim N^{-1} \| \nabla_x a_{i,1} \|_{L^\infty} \| v(t) \|_{L^2}.
\end{equation*}
Moreover, Strichartz estimates for $P$ are suitable to handle lower order perturbations as $\sum_{i=1}^d b_i(t,x) \partial_i + c(t,x)$. However, the above argument cannot deal with Schr\"odinger equations  with rougher coefficients as the free part is estimated in $L^2$ and not in $H^1$. The Duhamel integral does not gain one derivative. 

To obtain more inhomogeneous estimates for the wave equation, we record the following consequence of Theorem \ref{thm:HolderCoefficients} for solutions to
\begin{equation}
\label{eq:FreeSolutions}
\left\{ \begin{array}{cl}
\partial_t^2 u &= P u, \quad (t,x) \in \R \times \R^d, \\
u(0) &= f, \quad \partial_t u(0) = g.
\end{array} \right.
\end{equation}
 combined with conservation of energy \eqref{eq:Energy}.
\begin{corollary}
Let $d \geq 2$, $s \in (0,1)$, $\sigma = 1-s$, and $a_i \in C^{\infty}(\R)$, $i=1,\ldots,d$ satisfy \eqref{eq:Ellipticity}, and let $(\rho,p,q,d)$ be a wave Strichartz pair. Suppose that $u$ is a smooth solution to \eqref{eq:FreeSolutions}. Then, we find the following estimate to hold:
\begin{equation*}
\sup_{N} N^{1-\rho-\frac{\sigma}{p}} \| S_N u \|_{L^p([0,T],L^q)} \lesssim  \| f \|_{\dot{H}^1(\R^d)} + \| g \|_{L^2(\R^d)}
\end{equation*}
with implicit constant depending only on $T$, $\| a_i \|_{\dot{B}^s_{\infty,1}}$, $(\rho,p,q,d)$, and the ellipticity constants.
\end{corollary}

As a consequence of trading derivatives, we find the following homogeneous Strichartz estimate for the half-wave equation:
\begin{equation}
\label{eq:RoughHalfWaveStrichartzEstimates}
\| |D|^{-\rho-\frac{\sigma}{p}} e^{itL^{\frac{1}{2}}} f \|_{L^p([0,T],L^q(\R^d))} \lesssim \| f \|_{L^2(\R^d)}.
\end{equation}
In fact, for the sharp Strichartz estimates, it is easy to see that it suffices to trade less than one derivative. For non-sharp pairs one reduces to the sharp case first by Sobolev embedding. We prove inhomogeneous estimates by using again the taylored version of the Christ--Kiselev Lemma \ref{lem:ChristKiselev}.

\begin{proposition}
\label{prop:InhomogeneousWaveEstimatesHoelder}
Let $d \geq 2$, $s \in (0,1)$, $\sigma = 1-s$, and $(\rho,p,q,d)$, $(\tilde{\rho},\tilde{p},\tilde{q},d)$ be two wave Strichartz pairs with $p > \tilde{p}'$. Suppose that $a_i \in C^\infty(\R)$ satisfies \eqref{eq:Ellipticity} and $u$ solves
\begin{equation}
\label{eq:InhomogeneousRoughEquation}
\left\{ \begin{array}{cl}
\partial_t^2 u - \sum_{i=1}^d \partial_{x_i} (a(x_i) \partial_{x_i} u) &= F, \quad (t,x) \in [0,T] \times \R^d \\
u(0) &= f, \quad \partial_t u(0) = g.
\end{array} \right.
\end{equation}
Then, we find the following estimate to hold for $\rho_1 > \rho + \frac{\sigma}{p}$ and $\rho_2 > \rho + \frac{\sigma}{p}$:
\begin{equation}
\label{eq:InhomogeneousRoughEstimate}
\| \langle D \rangle^{1-\rho_1} u \|_{L^p([0,T],L^q(\R^d))} \lesssim \| f \|_{H^1(\R^d)} + \| g \|_{L^2(\R^d)} + \| \langle D \rangle^{\rho_2 } F \|_{L^{\tilde{p}'}([0,T],L^{\tilde{q}'}(\R^d))}
\end{equation}
with implicit constant only depending on $T$, $\| a_i \|_{\dot{C}^{0,1}}$, and the ellipticity constants.
\end{proposition}
\begin{proof}
We write
\begin{equation}
\label{eq:DuhamelRoughSolution}
u(t) = \cos(t L^{\frac{1}{2}}) f + \frac{\sin(t L^{\frac{1}{2}})}{L^{\frac{1}{2}}} g + \int_0^t \frac{\sin((t-s) L^{\frac{1}{2}})}{L^{\frac{1}{2}}} F(s) ds.
\end{equation}
The homogeneous components are estimated by \eqref{eq:RoughHalfWaveStrichartzEstimates}. We turn to the forcing term: Consider
\begin{equation*}
T: L^2(\R^d) \to L^p([0,T],L^q(\R^d)), \quad f \mapsto \langle D \rangle^{-\rho_1 } e^{it L^{\frac{1}{2}}} f
\end{equation*}
whose boundedness follows from the homogeneous estimates \eqref{eq:RoughHalfWaveStrichartzEstimates}. The dual operator with respect to the $L^2$-scalar product is given by
\begin{equation*}
T^* : L^{p'}([0,T],L^{q'}(\R^d)) \to L^2(\R^d), \quad F \mapsto \int_0^T e^{-is L^{\frac{1}{2}}} \langle D \rangle^{-\rho_1} F(s) ds.
\end{equation*}
We consider the composition with different exponents
\begin{equation*}
\begin{split}
TT^* : L^{\tilde{p}'}([0,T],L^{\tilde{q}'}(\R^d)) &\to L^p([0,T],L^q(\R^d)), \\
F &\mapsto \langle D \rangle^{-\rho_1} \int_0^T e^{i(t-s)L^{\frac{1}{2}}} (\langle D \rangle^{-\rho_2} F)(s) ds,
\end{split}
\end{equation*}
which is also bounded. Since by assumption $\tilde{p}' < p$, we can invoke Lemma \ref{lem:ChristKiselev} to find the bound
\begin{equation*}
\| \langle D \rangle^{-\rho_1} \int_0^t e^{i(t-s)L^{\frac{1}{2}}} F(s) ds \|_{L^p([0,T],L^q(\R^d))} \lesssim \| \langle D \rangle^{\tilde{\rho}+\frac{\sigma}{\tilde{p}}} F \|_{L^{\tilde{p}'}([0,T],L^{\tilde{q}'}(\R^d))}.
\end{equation*}
This yields the claim after trading $|D|$ to $|D_L|$.
\end{proof}

\subsection{Applications to nonlinear equations}
\label{subsection:NonlinearEquations}
With the usual Strichartz estimates at disposal, it is straightforward to prove well-posedness for a large class of equations with power-type nonlinearity. Here we consider as an example Schr\"odinger equations with power type nonlinearity in $L^2(\R^d)$:
\begin{equation}
\label{eq:NonlinearSEQ}
\left\{ \begin{array}{cl}
i \partial_t u + \sum_{i=1}^d \partial_{x_i} (a_i(x_i) \partial_{x_i}) u &= \mu |u|^{p-1} u , \quad (t,x) \in \R \times \R, \\
u(0) &= u_0 \in L^2(\R^d)
\end{array} \right.
\end{equation}
with $d \geq 1$ and $a_i \in BV(\R)$ satisfying \eqref{eq:Ellipticity} and $Var(\log(a_i)) < 2 \pi$. The results in the constant-coefficient case are due to Tsutsumi \cite{Tsutsumi1987}. For further explanation, we refer to \cite[Section~3.3]{Tao2006}. The following holds by substituting the free Strichartz estimates for the Schr\"odinger equation with the estimates from Theorems \ref{thm:GlobalEstimates} and \ref{thm:GlobalInhomogeneousStrichartzEstimates}:
\begin{theorem}[$L^2$-well-posedness]
Let $1<p<1+\frac{4}{d}$ and $\mu \in \{-1;1\}$. Then, \eqref{eq:NonlinearSEQ} is analytically globally well-posed in $L^2(\R^d)$ in the subcritical sense. If $p = 1 + \frac{4}{d}$, then \eqref{eq:NonlinearSEQ} is analytically locally well-posed in the critical sense.
\end{theorem}
We remark that also the $H^s$-theory extends as long as it is valid to trade derivatives $\| |D|^s f \|_{L^q} \sim \| |D_L|^{s} f \|_{L^q}$.

\section{Spectral multiplier estimates and Bochner--Riesz means}
\label{section:BochnerRiesz}
In the following we consider the non-negative self-adjoint operator
\begin{equation}
\label{eq:SchroedingerOperatorBV}
L = - \sum_{i=1}^d \partial_{x_i} (a_i(x_i) \partial_{x_i} )
\end{equation}
in $L^2(\R^d)$ for $a_i \in BV(\R)$ satisfying \eqref{eq:Ellipticity} and $Var(\log(a_i)) < 2\pi$. We derive consequences of the dispersive properties worked out in Proposition \ref{prop:GlobalDispersiveEstimate} for spectral restriction, multiplier estimates, and Bochner-Riesz means. The results follow from the analysis of Chen \emph{et al.} \cite{ChenOuhabazSikoraYan2016} for $L$ being a self-adjoint operator on a doubling metric space. Since $L$ is self-adjoint, it admits a spectral resolution $E_L(\lambda)$, and for $F:[0,\infty) \to \C$ a bounded Borel function, the operator
\begin{equation*}
F(L) = \int_{0}^\infty F(\lambda) dE_L(\lambda)
\end{equation*}
defines an $L^2$-bounded operator by the spectral theorem. A special case are the Bochner-Riesz means: We let
\begin{equation*}
S^\delta_R(\lambda) = (1-\lambda/R^2)^\delta_+
\end{equation*}
for $\delta > 0$ and $R > 0$. $S_R^\delta(L)$ is referred to as Bochner--Riesz mean of order $\delta$ corresponding to $L$. In the classical case $L=-\Delta$, the Bochner--Riesz conjecture states that $S^\delta_R(L):L^p(\R^d) \to L^p(\R^d)$ is bounded provided that $\delta > \max\big( d\big| \frac{1}{2} - \frac{1}{p} \big| - \frac{1}{2}, 0 \big)$ for $1 \leq p \leq \infty$ and $d \geq 2$. Note that for $p=2$, $\delta = 0$ is trivially admissible; see also \cite{Fefferman1970}. The conjecture was verified for $d=2$ by Carleson--Sj\"olin \cite{CarlesonSjoelin1972} and H\"ormander \cite{Hoermander1973}. By using the Stein--Tomas restriction theorem and finite speed of propagation, Fefferman verified the conjecture for $d \geq 3$ and $\max(p,p') \geq \frac{2d+2}{d-1}$ (note the self-duality). Further progress was closely tied to work on the restriction conjecture \cite{Lee2004,BourgainGuth2011,GuthHickmanIliopoulou2019}. In terms of the spectral measure $dE_{\sqrt{- \Delta}}(\lambda)$ the classical Stein--Tomas restriction theorem is equivalent to
\begin{equation}
\label{eq:SteinTomasRestrictionTT^*}
\| dE_{\sqrt{-\Delta}}(\lambda) \|_{p \to p'} = \Vert \frac{\lambda^{d-1}}{(2 \pi)^d} R_\lambda^* R_\lambda \|_{p \to p'} \leq C \lambda^{d\big( \frac{1}{p} - \frac{1}{p'} \big) - 1}
\end{equation}
for all $p \in [1,(2d+2)/(d+3)]$. In the first part of Chen \emph{et al.} \cite{ChenOuhabazSikoraYan2016} spectral multiplier estimates, which are sharp in general, were derived based on the following Stein--Tomas restriction condition:
\begin{equation*}
\big( \text{ST} \big)^q_{p,s} \quad \| F(\sqrt{L}) \|_{p \to s} \leq C R^{d\big( \frac{1}{p} - \frac{1}{s} \big)} \| \delta_R F \|_{L^q}.
\end{equation*}
We denote dilations by $\delta_R F(x) = F(Rx)$. These estimates in turn can be derived from \eqref{eq:SteinTomasRestrictionTT^*}. We have the following as an instance of \cite[Proposition~II.4]{ChenOuhabazSikoraYan2016}. Note that the required smoothing estimates
\begin{equation}
\label{eq:SmoothingEstimate}
\| \exp(-tL) \|_{p \to \frac{2d}{d+2}} \leq K t^{-\frac{d}{2} \big( \frac{1}{p} - \frac{d+2}{2d} \big)}
\end{equation}
for all $t>0$ and $1 \leq p \leq \frac{2d}{d+2}$ follow from the pointwise heat kernel estimate \eqref{eq:HeatKernelEstimate}.

\begin{proposition}
\label{prop:SpectralRestrictionSchroedingerOperator}
Let $d \geq 3$, and $L$ be defined as in \eqref{eq:SchroedingerOperatorBV}. Then, for all $1 \leq p \leq \frac{2d}{d+2}$ and $\lambda \geq 0$, the following estimate holds:
\begin{equation}
\label{eq:SpectralRestrictionEstimate}
\| d E_{\sqrt{L}}(\lambda) \|_{p \to p'} \leq C\lambda^{d \big( \frac{1}{p} - \frac{1}{p'} \big) - 1}.
\end{equation}
\end{proposition}
As a consequence of \eqref{eq:SpectralRestrictionEstimate} we derive the Stein--Tomas restriction (cf. \cite[p.~267]{ChenOuhabazSikoraYan2016})
\begin{equation*}
\begin{split}
\| F(\sqrt{L}) \|_{p \to p'} &= \big\| \int_0^R F(\lambda) dE_{\sqrt{L}} (\lambda) \|_{p \to p'} \leq C \int_0^R |F(\lambda)| \lambda^{d \big(\frac{1}{p} - \frac{1}{p'}\big) - 1} d\lambda \\
&\leq C R^{d \big( \frac{1}{p} - \frac{1}{p'} \big) } \| \delta_R F \|_1.
\end{split} 
\end{equation*}
By a $TT^*$-argument this is $(\text{ST})^2_{p,2}$. This yields the following spectral multiplier estimates and endpoint estimates for Bochner--Riesz means by \cite[Theorem~II.6]{ChenOuhabazSikoraYan2016} and \cite[Theorem~I.24]{ChenOuhabazSikoraYan2016}.
\begin{theorem}
Let $L$ be as in \eqref{eq:SchroedingerOperatorBV}, $d \geq 3$, and $p \in [1,2d/(d+2)]$. For each bounded Borel function $F$ such that $\sup_{t >0} \| \eta \delta_t F \|_{W^{\beta,2}} < \infty$ for some $\beta > \max\big( \{ d \big( \frac{1}{p} - \frac{1}{2} \big), \frac{1}{2} \} \big)$ and non-trivial $\eta \in C^\infty_c(0,\infty)$, $F(\sqrt{L})$ is bounded on $L^r(\R^d)$ for all $r \in (p,p')$. We find the following spectral multiplier estimate to hold:
\begin{equation*}
\| F(\sqrt{L}) \|_{r \to r} \leq C_\beta (\sup_{t >0} \| \eta \delta_t F \|_{W^{\beta,2}} + |F(0)| \big).
\end{equation*}
Furthermore, for $\max(p,p') > \frac{2d}{d-2}$, the Bochner--Riesz means $S_R^{\delta(p)}(L):L^p(\R^d) \to L^p(\R^d)$ are bounded and satisfy weak endpoint bounds uniformly in $R$ with $\delta(p)= \max(d \big| \frac{1}{2} - \frac{1}{p} \big| - \frac{1}{2}, 0)$.
\end{theorem}
We turn to maximal Bochner--Riesz operators $S^\alpha_*(L)$ defined by
\begin{equation}
\label{eq:MaximalBochnerRiesz}
S^\alpha_*(L) f(x) = \sup_{R > 0} |S^\alpha_R(L) f(x)|.
\end{equation}
These were explored in the general context described above by Chen \emph{et al.} in \cite{ChenLeeSikoraYan2020}. \cite[Theorem~A]{ChenLeeSikoraYan2020} yields the following:
\begin{theorem}
Let $L$ be as in \eqref{eq:SchroedingerOperatorBV}, and $d \geq 3$. Then the maximal Bochner--Riesz operator $S^\alpha_*(L)$ is bounded on $L^p(\R^d)$ provided that
\begin{equation}
\label{eq:RangeMaximalOperator}
2 \leq p < \frac{2d}{d-2}, \text{ and } \alpha > \max( d \big( \frac{1}{2} - \frac{1}{p} \big) - \frac{1}{2},0).
\end{equation}
In particular, we find pointwise convergence of the Bochner--Riesz means to hold:
\begin{equation}
\label{eq:PointwiseConvergence}
\lim_{R \to \infty} S^\alpha_R(L) f(x) = f(x) \text{ a.e.}
\end{equation}
\end{theorem}

In \cite{SikoraYanYao2018} Sikora--Yan--Yao considered Bochner--Riesz estimates of negative index and spectral multiplier estimates $L^p(\R^d) \to L^q(\R^d)$. The results were likewise obtained in the general context described above, hinging on heat kernel and Tomas--Stein restriction estimates. We consider for $\alpha > -1$
\begin{equation}
\label{eq:BochnerRieszOperatorNegativeIndex}
S^\alpha_R(\sqrt{L}) = \frac{1}{\Gamma(\alpha+1)} \big( 1 - \frac{\sqrt{L}}{R} \big)^\alpha_+.
\end{equation}
For $\alpha = -1$, we set $S_R^{-1}(\sqrt{L}) = R^{-1} dE_{\sqrt{L}}(R)$. This is based on the distributional limit (cf. \cite[Eq.~(3.2.17')]{Hoermander2003}):
\begin{equation*}
\lim_{\alpha \downarrow -1} \frac{1}{\Gamma(\alpha+1)} x_+^\alpha = \delta(x).
\end{equation*}
Let $L$ be as above \eqref{eq:SchroedingerOperatorBV}. By the pointwise Gaussian estimates noted in \eqref{eq:HeatKernelEstimate}, Davies--Gaffney estimates of second order follow
\begin{equation}
\label{eq:DGEstimates}
\| P_{B(x,t^{\frac{1}{2}})} e^{-t L} P_{B(y,t^{\frac{1}{2}})} \|_{2 \to 2} \leq C \exp ( - \frac{c|x-y|^2}{t})
\end{equation}
for all $t>0$, and $x,y \in \R^d$. $P_E f(x) = \chi_E(x) f(x)$ denotes multiplication with an indicator function. Moreover, the condition
\begin{equation}
\label{eq:Gp2}
\| e^{-t^2 L} \|_{p \to 2} \leq C t^{d\big(\frac{1}{2}-\frac{1}{p}\big)}
\end{equation}
holds for all $x \in \R^d$ and $t > 0$.\\
The Bochner--Riesz estimates of negative index investigated in \cite{SikoraYanYao2018} for operators satisfying \eqref{eq:DGEstimates} and \eqref{eq:Gp2} read
\begin{equation*}
(\text{BR}^\alpha_{p,q}) \quad \| S^\alpha_R(\sqrt{L}) \|_{p \to q} \leq C R^{d \big( \frac{1}{p} - \frac{1}{q} \big)}.
\end{equation*}
Thus, the Stein--Tomas restriction estimates \eqref{eq:SpectralRestrictionEstimate} showed in Proposition \ref{prop:SpectralRestrictionSchroedingerOperator} corresponds to $(\text{BR}^\alpha_{p,p'})$. Sikora--Yan--Yao \cite[Theorem~3.9]{SikoraYanYao2018} proved that Bochner--Riesz estimates of negative index imply the following spectral multiplier estimates:
\begin{theorem}
\label{thm:SpectralMultiplierMixedIndices}
Let $L$ be as in \eqref{eq:SchroedingerOperatorBV}. Suppose that $(\text{BR}^\alpha_{p,q})$ holds for $\alpha \geq -1$ and $1<p<q<\infty$. Let $p \leq r \leq s \leq q$, and $\beta > d \big( \frac{1}{p} - \frac{1}{r} \big) + d \big( \frac{1}{s} - \frac{1}{q} \big) + \alpha + 1$, $\text{supp}(F) \subseteq [1/4,4]$, and $F \in W^{\beta,1}(\R)$, the operator $F(t \sqrt{L})$ is bounded from $L^r(\R^d) \to L^s(\R^d)$. Moreover, the following estimate holds:
\begin{equation}
\label{eq:SpectralMultiplierEstimate}
\sup_{t > 0} t^{d\big( \frac{1}{r} - \frac{1}{s} \big)} \| F(t \sqrt{L}) \|_{r \to s} \leq C \| F \|_{W^{\beta,1}(\R)}.
\end{equation} 
\end{theorem}
For a second result on spectral multipliers, we recall the definition of the Weyl--Sobolev norm: The distributions
\begin{equation}
\label{eq:HomogeneousDistributions}
\chi_{\pm}^\alpha = \frac{x_{\pm}^\alpha}{\Gamma(\alpha+1)} \quad \Re \alpha > -1.
\end{equation}
can be extended to arbitrary index $\nu \in \C$ by respecting the recursion relation of the derivatives (cf. \cite[p.~3087]{SikoraYanYao2018}). For $\text{supp}(F) \subseteq [0,\infty)$, we then define the Weyl fractional derivative of $F$ of order $\nu$ by
\begin{equation*}
F^{(\nu)} = F * \chi_-^{-\nu-1}, \quad \nu \in \C
\end{equation*}
and the Weyl--Sobolev norm by
\begin{equation*}
\| F \|_{WS^{\nu,p}} = \| F \|_p + \| F^{(\nu)} \|_p.
\end{equation*}
For $1<p<\infty$ and $\nu \geq 0$, the Weyl--Sobolev norm is equivalent to the usual Sobolev norm $\| F \|_{W^{\nu,p}} \sim \| F \|_{WS^{\nu,p}}$ whereas for $p=1$ we have
\begin{equation*}
\| F \|_{WS^{\nu,1}} \leq C_\varepsilon \| F \|_{W^{\nu+\varepsilon,1}}
\end{equation*}
for any $\varepsilon > 0$ (cf. \cite[Lemma~3.7]{SikoraYanYao2018}). The following result \cite[Proposition~3.8]{SikoraYanYao2018} is not restricted to dyadically supported multipliers:
\begin{proposition}
Suppose that $(\text{BR}^\alpha_{p,q})$ holds for some $\alpha \geq -1$ and $1 \leq p < q \leq \infty$. Then, for every $\varepsilon>0$, there exists a constant $C_\varepsilon$ such that for any $R>0$ and all Borel functions $F$ with $\text{supp}(F) \subseteq [R/2,R]$, we find the estimate
\begin{equation*}
\| F(\sqrt{L}) \|_{p \to q} \leq C R^{d\big( \frac{1}{p} - \frac{1}{q} \big)} \| \delta_R F \|_{WS^{\alpha+1,1}(\R)}
\end{equation*}
to hold.
\end{proposition}

Moreover, Bochner--Riesz estimates of index $-1$ yield Bochner--Riesz estimates of higher order with more admissible indices (cf. \cite[Theorem~3.12]{SikoraYanYao2018}). We omit the details.

\section*{Data availability statement}

Data sharing not applicable to this article as no datasets were generated or analysed during the current study.

\section*{Acknowledgements}
Funded by the Deutsche Forschungsgemeinschaft (DFG, German Research Foundation)  Project-ID 258734477 – SFB 1173.

\end{document}